\documentclass[12pt,a4paper,intlimits,sumlimits]{amsart}

\usepackage{amsmath, amsthm, mathtools}
\usepackage{amsfonts}
\usepackage[alphabetic,nobysame]{amsrefs}
\usepackage{graphicx}
\usepackage{framed}
\usepackage{amssymb}
\usepackage{esvect}
\usepackage{xcolor}
\usepackage{eucal}
\usepackage{mathrsfs}
\usepackage{hyperref}
\usepackage[english]{babel}
\usepackage{mathrsfs}
\usepackage{esint}
\usepackage{accents}

\def \N {\mathbb{N}}
\def \R {\mathbb{R}}

\theoremstyle{definition}
\newtheorem{definition}{Definition}[section]

\newtheorem{remark}[definition]{Remark}

\theoremstyle{plain}
\newtheorem{theorem}[definition]{Theorem}
\newtheorem{proposition}[definition]{Proposition}
\newtheorem{lemma}[definition]{Lemma}
\newtheorem{corollary}[definition]{Corollary}

\numberwithin{equation}{section}

\usepackage{geometry}
\renewcommand{\epsilon}{\varepsilon}
\newcommand{\e}{\varepsilon}

\renewcommand{\le}{\leqslant}

\renewcommand{\ge}{\geqslant}

\allowdisplaybreaks

\title[Nonlocal operators and integrable data]{Nonlocal operators in divergence form \\ and existence theory
for integrable data}

\author[D. Arcoya, S. Dipierro, E. Proietti Lippi, C. Sportelli, E. Valdinoci]{David Arcoya, Serena Dipierro, Edoardo Proietti Lippi, \\Caterina Sportelli and Enrico Valdinoci}

\address{David Arcoya: Departamento de Análisis Matemático, Universidad de Granada, 18071 Granada, Spain}
\email{darcoya@ugr.es}

\address{Serena Dipierro: Department of Mathematics and Statistics, The University of Western Australia, 35 Stirling Highway, Crawley, Perth, WA 6009, Australia}
\email{serena.dipierro@uwa.edu.au}

\address{Edoardo Proietti Lippi: Department of Mathematics and Statistics, The University of Western Australia, 35 Stirling Highway, Crawley, Perth, WA 6009, Australia}
\email{edoardo.proiettilippi@uwa.edu.au}

\address{Caterina Sportelli: Department of Mathematics and Statistics, The University of Western Australia, 35 Stirling Highway, Crawley, Perth, WA 6009, Australia}
\email{caterina.sportelli@uwa.edu.au}

\address{Enrico Valdinoci: Department of Mathematics and Statistics, The University of Western Australia, 35 Stirling Highway, Crawley, Perth, WA 6009, Australia}
\email{enrico.valdinoci@uwa.edu.au}

\begin{document}

\maketitle

\begin{abstract}
We present an existence and uniqueness result for weak solutions of Dirichlet boundary value problems governed by a nonlocal operator in divergence form and
in the presence of a datum which is assumed to belong only to~$L^1(\Omega)$ and to be suitably dominated.

We also prove that the solution that we find converges, as~$s\nearrow 1$, to a solution of the local counterpart problem,
recovering the classical result as a limit case.
This requires some nontrivial customized uniform estimates
and representation formulas,  given that the datum is only in~$L^1(\Omega)$ and therefore the usual regularity theory cannot be leveraged to our benefit in this framework.

The limit process uses a
nonlocal operator, obtained as an affine transformation
of a homogeneous kernel,
which recovers, in the limit as~$s\nearrow 1$,  every classical operator in divergence form.  
\end{abstract}

\tableofcontents

\section{Introduction}

\subsection{Nonlocal existence theory for integrable data}
A long-established line of research focuses on the construction of solutions for suitable elliptic equations. In the traditional setting, the data of the equation are typically required to belong to some~$L^p$-space with~$p > 1$, as this allows for the application of regularity theory and standard functional analysis techniques.  

When~$p = 1$, these methods are typically inaccessible. For instance, as is well known, regularity theory fails in the presence of merely integrable data (see e.g. Sections~4.1 and~5.1 in~\cite{MR4784613}). Consequently, no general existence theory for second-order elliptic equations is available in this context.

However, the results presented in~\cite{MR3304596, MR3760750} established an existence theory for equations
in divergence form under certain structural assumptions on the equation. These assumptions permit the implementation of a tailored uniform approximation, which produces the desired solution via a limit procedure.

One of the goals of this article is to extend this existence theory to the nonlocal setting.
To this end, we consider an interaction kernel~$K:\R^n\times\R^n\to[0,+\infty]$ such that
\begin{equation}\label{QrehtgDWFfbg0jolwfdv:Qwfgb}
\frac{c\,\chi_{[0,\varrho)}\big(|x-z|\big)}{|x-z|^{n+2s}}\le K(x,z)\le \frac{C}{|x-z|^{n+2s}},
\end{equation}
for some~$c$, $C\in(0,+\infty)$, $\varrho\in(0,+\infty]$ and~$s\in(0,1)$, where, as usual, $\chi_{[0,\varrho)}$ denotes the characteristic function of $[0,\varrho)$, that is
$$ \chi_{[0,\varrho)}\big(|x-z|\big):=\begin{cases} 1&{\mbox{ if }}|x-z|\in [0,\varrho),\\0&{\mbox{ otherwise.}}\end{cases}$$
To be compatible with the divergence structure, it is convenient to assume the symmetry
condition
\begin{equation}\label{SyMQWD:K} K(x, z)=K(z, x) \ \mbox{ for any } x, z\in\R^n,
\end{equation}
see e.g.~\cite[Section 2.5]{ROSOTON} and the references therein
for more information about nonlocal operators with a divergence structure.

The nonlocal operator associated with the interaction kernel~$K$ can thus be written in the poinwise form
\begin{equation}\label{qsdc:OPK}\begin{split}
\mathcal L_K u(x)&:= \mbox{P.V. }\int_{\R^n} \big(u(x) - u(z)\big)\,K(x,z)\, dz\\
&= \mbox{P.V. }\int_{\R^n} \big(u(x) - u(x-y)\big)\,K(x,x-y)\, dy\\
&:=\lim_{\varepsilon\searrow0}\int_{\R^n\setminus B_\varepsilon} \big(u(x) - u(x-y)\big)\,K(x,x-y)\, dy,
\end{split}\end{equation}
where P.V. stands for the principal value notation.

We observe that the kernel~$\R^n\ni y\mapsto K(x,x-y)$ is not assumed to be even in~$y$.
This lack of symmetry prevents, in principle, the pointwise calculation in~\eqref{qsdc:OPK} when $s\in\left[\frac12,1\right)$, even for smooth functions. Thus, when $s\in\left[\frac12,1\right)$, to make sense of the above setting it is customary
to assume an ``almost symmetry'' condition for the kernel with respect to the variable~$y$.
Specifically, if one wishes to use the pointwise form of the operator in~\eqref{qsdc:OPK}
it comes in handy to assume that\footnote{Other technical
conditions can be used to replace~\eqref{COND:qsdfdv}, see e.g.~\cite[Section~2.5.2]{ROSOTON}
for an extensive discussion.}
\begin{equation}\label{COND:qsdfdv}
\sup_{x\in\R^n}\int_{B_1} |y|\,\big|K(x,x+y)-K(x,x-y)\big|\,dy<+\infty,
\end{equation}
see in particular Lemma~\ref{apkdjoscvndfSHDKB<qrg4596uyh43io}.

In this paper, however, we will not assume condition~\eqref{COND:qsdfdv}
to prove our existence results in Theorems \ref{main2} and \ref{thmminimo}, since they rely only on weak formulations. However, it will be convenient to pass through condition~\eqref{COND:qsdfdv} when dealing with limit properties.
\medskip

In the rest of this paper, we will denote by~$\Omega$
an open bounded subset of~$\R^n$ with Lipschitz boundary.
Also, we will denote by~$H^s_0(\Omega)$ the space of functions in~$H^s(\R^n)$ such that~$u=0$ a.e. in~$\R^n\setminus\Omega$
(see Section~\ref{SEC4iueryihglsdg} for the details related to the functional setting).

In this framework, our existence result reads as follows:

\begin{theorem}\label{main2}
Let~$h:\R\to\R$. Suppose that
\begin{equation}\label{ginftydef}
h \mbox{ is continuous, odd and strictly increasing}
\end{equation}
and let
\begin{equation}\label{ginftydef2}
\gamma:=\lim_{t\to+\infty} h(t)\in(0,+\infty].
\end{equation}

Let~$a$, $f\in L^1(\Omega)$. Assume that
\begin{equation}\label{GQcond}
\mbox{there exists $Q\in (0, \gamma)$ such that } |f(x)|\le Q\, a(x) \mbox{ for a.e. }x\in \Omega.
\end{equation}

Then, there exists a unique weak solution~$u\in H^s_0(\Omega)\cap L^\infty(\Omega)$ of
\begin{equation}\label{davidnonlocal2}
\begin{cases}
\mathcal L_K u + a h(u) = f &\mbox{ in } \Omega,\\
u=0 &\mbox{ in } \R^n\setminus\Omega.
\end{cases}
\end{equation}

In addition, we have that
\begin{align}\label{GLinftyu}
&\|u\|_{L^\infty(\Omega)}\le h^{-1}(Q)\\ \label{seminormaus}
\mbox{ and }\quad & \iint_{\R^n\times\R^n} \big(u(x) - u(z)\big)^2\,K(x,z)\,dx\, dz
\le 2h^{-1}(Q) \,\big(\|f\|_{L^1(\Omega)} + Q \|a\|_{L^1(\Omega)}\big) .
\end{align}
\end{theorem}

We stress that
the result in Theorem~\ref{main2} is new in the nonlocal setting, even in the special case in which~$K(x,z)$
is, up to multiplicative constants, the homogeneous kernel~$|x-z|^{-n-2s}$ (in this situation,
the operator~\eqref{qsdc:OPK} reduces to the fractional Laplacian).\medskip

To prove Theorem~\ref{main2}, the following result, obtained by a minimization argument,
will be also helpful:

\begin{theorem}\label{thmminimo}
Let~$h:\R\to\R$ satisfy~\eqref{ginftydef}.
Let~$\eta\in L^\infty(\Omega)$ with~$\eta\ge 0$ a.e. in~$\Omega$, and~$\zeta\in L^2(\Omega)$.

Then, there exists a weak solution~$u\in H^s_0(\Omega)$ of
\begin{equation}\label{problemaGxuu}
\begin{cases}
\mathcal L_K u + \eta h(u) = \zeta &\mbox{ in } \Omega,\\
u=0 &\mbox{ in } \R^n\setminus\Omega.
\end{cases}
\end{equation}
\end{theorem}

\subsection{Back to the classical case (from~$M$ to~$A$)} The nonlocal result obtained in Theorem~\ref{main2}
is also useful to shed new light on the classical case, recovering
it in the limit. This limit procedure is not standard and will require,
on the one hand, the selection of ``the most natural'' fractional operator
that recovers a given second-order elliptic operator in divergence form,
and, on the other hand, a number of ad-hoc inequalities allowing one to
implement this asymptotic technique.
In this sense, we believe that the methodology introduced here can be
advantageous in several situations, not only to recover known results
in partial differential equations as a limit case of nonlocal equations (which
has also the benefit of providing a unified scenario comprising both classical
and fractional problems), but also potentially to obtain new results in 
partial differential equations by using techniques coming from the fractional world.
This approach has been used, for instance,  in the regularity theory ``by approximation" performed in~\cite{MR2781586}, in the asymptotic regularity theories of nonlocal minimal surfaces carried out in~\cite{MR3107529, MR4116635, 2023arXiv230806328C},
as well as in the construction of nonlocal catenoids in~\cite{MR3798717}.
\medskip

To implement our strategy, we recall that second-order operators in divergence form are typically written in the form
\begin{equation}\label{classicaldivergenceform}
-\operatorname{div}\big(A(x)\nabla u(x)\big)=
- \sum_{i, j=1}^n \frac{\partial}{\partial x_i}\left(A_{ij}(x) \frac{\partial u}{\partial x_j}(x) \right),
\end{equation}
where\footnote{As customary, the entries of a given matrix~$L$
will be denoted by~$L_{ij}$.} $A:\R^n\to\R^{n\times n}$ is 
symmetric, uniformly bounded and uniformly positive definite.
\medskip

A natural question in this context is whether or not it
is possible to introduce a nonlocal operator which recovers the structure of the operator in~\eqref{classicaldivergenceform} through an asymptotic limit procedure
(by focusing on the ``simplest'' possible structure of this nonlocal operator).
This plan of action has been presented in~\cite[Section~3]{MR3967804}, where it was proven that the operator in~\eqref{classicaldivergenceform} can be recovered starting from a ``master equation". 
However,  our approach demands more accuracy on the assumptions needed to make this procedure work, to enable us to offer a unified view of local and nonlocal cases and to recover the local case starting from the nonlocal one through a limit process.

The ``simplest possible'' approach in this framework consists in\footnote{Of course, when~$\varrho=+\infty$,
we have that~$B_\varrho=\R^n$.}
taking~$\varrho\in(0,+\infty]$
and~$M:\R^n\times B_\varrho\to\R^{n\times n}$
and in ``modulating'' a fractional interaction kernel by means of the matrix-valued function~$M$.

More precisely, we choose\begin{equation}\label{kerneldiv}
K(x, z):= \frac{c_{n, s}\,\chi_{[0,\varrho)}\big(|x-z|\big)}{ |M(z, x-z)(x-z)|^{n+2s}},
\end{equation} 
and we suppose that
there exist~$\alpha\ge\beta>0$ such that, for any~$x\in\R^n$, $y\in B_\varrho$ and $\xi\in\R^n$,
\begin{equation}\label{matdefpos}
\beta |\xi|\le |M(x-y, y)\xi| \le  \alpha |\xi|.
\end{equation}
To ensure compatibility with the divergence form structure, we suppose that
\begin{equation}\label{structuralM}
M(x-y, y) = M(x, -y)\quad\mbox{ for any } (x, y)\in\R^n\times B_\varrho.
\end{equation}

This setting
produces a nonlocal operator defined as
\begin{equation}\label{LUoperator}\begin{split}
\mathcal L_{M,\varrho,s} u(x)&:= c_{n, s}\, \mbox{P.V. }\int_{B_\varrho} \frac{u(x) - u(x-y)}{|M(x-y, y) y|^{n+2s}} \,dy\\
&= c_{n, s}\, \mbox{P.V. }\int_{B_\varrho(x)} \frac{u(x) - u(z)}{|M(z, x-z) (x-z)|^{n+2s}} \,dz
.\end{split}\end{equation}
Here and in the rest of this paper,
the constant~$c_{n,s}$ is taken to be
\begin{equation}\label{defcNs}
c_{n,s}:=\frac{2^{2s}\,\Gamma\left(\frac{n+2s}{2}\right)}{\pi^{n/2}\,\Gamma(2-s)}\, s(1-s),
\end{equation}
where~$\Gamma$ denotes the Euler Gamma function. The choice of the constant~$c_{n,s}$ allows the recovering of the limit cases as~$s\searrow0$ and~$s\nearrow1$, see e.g.~\cite{GENTLE}.

In particular, if~$M$ is the identity matrix, the operator in ~\eqref{LUoperator} reduces to the fractional Laplacian. We also observe that the kernel in~\eqref{kerneldiv}
satisfies the structural condition~\eqref{QrehtgDWFfbg0jolwfdv:Qwfgb}, 
thanks to~\eqref{matdefpos}, and the symmetry condition in~\eqref{SyMQWD:K}, thanks to~\eqref{structuralM}.

We recall that the idea of considering affine transformations of nonlocal operators is particularly
fruitful also regarding the theory of fractional operators in nondivergence form and fully nonlinear equations
(compare, for example, our operator~$\mathcal L_{M,\varrho,s}$ with~$L_A$ in~\cite{MR3479063}), but the role of affine transformations
in nonlocal operators with divergence form has not been fully investigated yet.
\medskip

This nonlocal setting recovers the classical case as~$s\nearrow1$. More explicitly, we have that:

\begin{proposition}\label{Proplimito1}
Let~$s\in (0, 1)$ and~$u\in H^1(\R^n)$. Suppose that~$\varrho\in(0,+\infty]$ and~$M:\R^n\times B_\varrho\to\R^{n\times n}$
satisfy~\eqref{matdefpos} and \eqref{structuralM}. If there exists~$r\in(0,\varrho)$ such that
\begin{equation}\label{ASSUNZ:MM} \begin{split}&{\mbox{the map~$\R^n\times B_r\ni(x,
y)\longmapsto M(x,y)$ is continuous,}}\end{split}
\end{equation}
then, for any~$\varphi\in C^\infty_c(\R^n)$,
\begin{equation}\label{limitesto1distr}\begin{split}&
\lim_{s\nearrow 1}\frac{c_{n, s}}{2}\iint_{\R^n\times B_\varrho} \frac{(u(x)-u(x-y))(\varphi(x)-\varphi(x-y))}{|M(x-y, y) y|^{n+2s}} \,dx\, dy\\&\qquad= \sum_{i, j=1}^n \,\int_{\R^n} A_{ij}(x)\frac{\partial u}{\partial x_i}(x) \frac{\partial\varphi}{\partial x_j}(x)\, dx,\end{split}
\end{equation}
where\footnote{As customary, here and in the rest of this paper~$\omega_{n-1}$ denotes the surface area of the unit sphere in~$\R^n$.}
\begin{equation}\label{aijdefn}
A_{ij}(x) := \frac{n}{\omega_{n-1}}\int_{\mathcal S^{n-1}} \frac{\psi_i \psi_j}{|M(x, 0)\,\psi|^{n+2}} d\mathcal H_\psi^{n-1}.
\end{equation}

In addition, denoting by~$A(x)$ the matrix with entries~$A_{ij}(x)$, there exist~$\overline\alpha\ge\overline\beta>0$, depending only on~$\alpha$, $\beta$ and~$n$,
such that, for any~$x$, $\xi\in\R^n$,
\begin{equation*}
\overline\beta |\xi|^2\le A(x)\xi\cdot \xi \le \overline\alpha |\xi|^2.
\end{equation*}
\end{proposition}

The proof of Proposition~\ref{Proplimito1} will be presented in Appendix~\ref{SEC2}.

The spirit of this kind of asymptotic results is that,
given~$M:\R^n\times B_\varrho\to\R^{n\times n}$, which will be used in a nonlocal problem, one
constructs a matrix-valued function~$A$ as in~\eqref{aijdefn} corresponding to a classical second-order equation. 
This is useful when considering sequences of solutions~$u_s$ of fractional equations 
related to the operator~$\mathcal L_{M,\varrho,s}$ in order to
determine its limit as~$s\nearrow1$. In our setting,
in particular, this strategy leads to the following result:

\begin{theorem}\label{main3}
Let~$\varrho\in(0,+\infty]$ and~$M:\R^n\times B_\varrho\to\R^{n\times n}$ be a Lipschitz continuous matrix-valued function such that~\eqref{matdefpos} and 
\eqref{structuralM} 
hold true. 

Let~$h:\R\to\R$ satisfy~\eqref{ginftydef}. 
Let~$a$, $f\in L^1(\Omega)$ be such that~\eqref{GQcond} is satisfied.

Given~$s\in (0, 1)$,
let~$u_s\in H^s_0(\Omega)\cap L^\infty(\Omega)$ be the unique weak solution of
\begin{equation}\label{davidnonlocal2MS}
\begin{cases}
\mathcal L_{M,\varrho,s} u + a h(u) = f &\mbox{ in } \Omega,\\
u=0 &\mbox{ in } \R^n\setminus\Omega,
\end{cases}
\end{equation}
provided by Theorem~\ref{main2}.

Then, there exists a unique weak solution~$u_1\in H_0^1(\Omega)\cap L^\infty(\Omega)$ of
\begin{equation}\label{problemalocale}
\begin{cases}
-\operatorname{div}(A\nabla u) + a h(u) = f&\mbox{ in } \Omega,\\
u=0 &\mbox{ on } \partial\Omega,
\end{cases}
\end{equation}
where~$A$ is as in~\eqref{aijdefn},
and, for a.e.~$x\in\R^n$,
\[
\lim_{s\nearrow 1} u_s(x) = u_1(x).
\]

In addition, we have that
\begin{equation}\label{GLinftyuBIS}\begin{split}
&\|u_1\|_{L^\infty(\Omega)}\le h^{-1}(Q)\\ 
\mbox{ and }\quad & \|\nabla u_1\|_{L^2(\Omega)}^2
\le C\,h^{-1}(Q) \,\big(\|f\|_{L^1(\Omega)} + Q \|a\|_{L^1(\Omega)}\big) ,
\end{split}\end{equation}
where~$C>0$ depends only on~$n$, $\alpha$, $\beta$ and~$\Omega$.
\end{theorem}

The stability property showcased in Theorem~\ref{main3} is important since
it highlights a series of uniform estimates which go beyond the standard regularity theory
and presents a ``unified'' vision in which the classical case is actually
obtained as a byproduct of a nonlocal construction.

We observe that the existence of a classical solution~$u_1$ for the second-order equation~\eqref{problemalocale} was
established in~\cite[Theorem~2.1]{MR4145478}, working in a local framework. However, in Theorem~\ref{main3}
we deduce the existence of this solution as limit of solutions of nonlocal problems.
Indeed, the concrete applicability of Theorem~\ref{main3} leverages Theorem~\ref{main2}
to ensure the existence of solutions for all nonlocal problems,
then it relies on some tailor-made estimates which are independent of the parameter~$s\in (0, 1)$
in order to pass these solutions to the limit as~$s\nearrow1$, finally concluding that the limit function
obtained in this manner is actually a solution of the classical problem (we will see in the forthcoming
Theorem~\ref{fuweioifhqweio7564} that this strategy can also be adopted to construct a solution of any classical problem in divergence form).

\subsection{Back to the classical case (from~$A$ to~$M$)}
Interestingly, one can also modify the above strategy
in order to provide an independent proof 
of~\cite[Theorem~2.1]{MR4145478} by relying on a limit procedure.
This requires significant additional work, since,
while Proposition~\ref{Proplimito1} describes the limit matrix-valued function~$A$
in terms of a given matrix-valued function~$M$ used in a sequence of fractional problems,
to recover the classical case by a limit method one needs
some ``reverse engineering'': namely, given any elliptic matrix-valued function~$A$
one needs to construct a matrix-valued function~$M$ to be used in the sequence of fractional problems.
In a sense, one needs to ``invert'' the relation in~\eqref{aijdefn}
and, as detailed in the forthcoming equation~\eqref{0wqurj3o-4uy4ijp30to5Xmm0vb5vn},
such inversion is not unique (however, the explicit inversion formula
that we find allows for ``the simplest possible choice'' of~$M$
to recover~$A$).

To clarify the technicalities involved in this construction, let us first
present a result in linear algebra\footnote{As customary, given a matrix~$N\in\R^{n\times n}$, we denote by~$N^T$ the transpose matrix.
Also, $\delta_{ij}$ denotes the Kronecker symbol.} permitting to invert~\eqref{aijdefn}
at a given point:

\begin{lemma}\label{lemma3}
Let~$\lambda_1, \dots, \lambda_n>0$ and~$A\in\R^{n\times n}$ be defined as
\begin{equation}\label{AOLAMBDAOT}
A:=O\Lambda O^T,
\end{equation}
where~$O$, $\Lambda\in\R^{n\times n}$ satisfy
\begin{equation}\label{AOLAMBDAOT2}
OO^T = \operatorname{Id}\quad\mbox{ and }\quad \Lambda_{ij}:= \lambda_i\delta_{ij}.
\end{equation}
Let~$\overline\lambda:=\lambda_1\cdots\lambda_n$
and, for any~$i\in\{1,\dots, n\}$,
\begin{equation}\label{sigmai2}
\sigma_i:=\sqrt{\frac{\overline\lambda^{\frac{1}{n+2}}}{\lambda_i}}.
\end{equation}
Let~$\Sigma\in\R^{n\times n}$ with entries~$\sigma_i\delta_{ij}$ and set
\begin{equation}\label{MOSIGMAOT}
N_A:= O\Sigma O^T.
\end{equation}

Then, 
\begin{equation}\label{20ef-wdfuvjb9ogrbtg}
A_{ij} = \dfrac{n}{\omega_{n-1}}\int_{\mathcal S^{n-1}} \frac{\psi_i \,\psi_j}{|N_A\, \psi|^{n+2}}\, d\mathcal H^{n-1}_\psi.
\end{equation}

\end{lemma}

We recall that, by the Spectral Theorem (see e.g.~\cite[Theorem~5S, page~330]{STRANG4}), conditions~\eqref{AOLAMBDAOT} and~\eqref{AOLAMBDAOT2} just make it explicit that the matrix~$A$ is symmetric
with eigenvalues~$\lambda_1,\dots,\lambda_n$. Furthermore, the positivity of these eigenvalues is equivalent to the
positive definiteness of~$A$.

The inversion provided by Lemma~\ref{lemma3} needs to
be adapted to comprise the interaction setting appearing in the fractional case:
roughly speaking, while in the classical problem~\eqref{problemalocale} the matrix-valued function~$A$
depends only on the point~$x$, its nonlocal counterpart makes use
of the setting in~\eqref{LUoperator} and the corresponding matrix-valued function~$M$ depends on both~$x$ and~$y$
(or both~$x$ and~$z$, depending on the preferred form in which one can write the operator).
This is done according to the following result:

\begin{theorem}\label{THMFROMATOM}
Suppose that~$A:\R^n\to\R^{n\times n}$ is a continuous matrix-valued function.
Suppose also that, for all~$x\in\R^n$, the matrix~$A(x)$ is symmetric
and denote by~$\lambda_i(x)$ its eigenvalues.

Assume that there exist~$\undertilde\lambda$, $\widetilde\lambda>0$ such that, for any~$i\in\{1,\dots, n\}$ and for any~$x\in\R^n$, 
\begin{equation}\label{tildelambdaMAIN}
\undertilde\lambda\le\lambda_i(x)\le \widetilde\lambda.
\end{equation}

Then, there exists $M:\R^n\times\R^n\to\R^{n\times n}$
satisfying~\eqref{matdefpos}, \eqref{structuralM} and~\eqref{aijdefn} (with~$\varrho=+\infty$).

The bounds in~\eqref{matdefpos} depend only on~$n$, $\undertilde\lambda$ and~$\widetilde\lambda$.

The explicit expression of~$M$ has the form
\begin{equation}\label{0wqurj3o-4uy4ijp30to5Xmm0vb5vn}
M(x, y):= \dfrac{N_{A(x)}+N_{A(x+y)}}{2} + H(y), 
\end{equation}
where the notation in~\eqref{MOSIGMAOT} has been used,
and~$H:\R^n\to\R^{n\times n}$ is any bounded and continuous matrix-valued function,
with~$H(y)$ positive semidefinite uniformly in~$y$, satisfying
\begin{equation}\label{ipotesiP:NEL}
H(0)=0\quad\mbox{ and }\quad H(y)=H(-y) \ \mbox{ for any~$y\in\R^n$.}
\end{equation}
\end{theorem}

We point out that condition~\eqref{tildelambdaMAIN} is equivalent to assume that~$A(x)$ is bounded and strictly positive definite, uniformly in~$x$. We have however
stated this condition explicitly to clarify the dependence of the structural constants on the quantities~$\undertilde\lambda$ and~$ \widetilde\lambda$.

The representation formula~\eqref{0wqurj3o-4uy4ijp30to5Xmm0vb5vn} is interesting,
since, given an elliptic matrix~$A$ for the second-order problem in divergence form,
it selects the ``simplest'' possible nonlocal operator in divergence form
that recovers the classical case in the limit: in particular, the choice~$H:=0$ in~\eqref{0wqurj3o-4uy4ijp30to5Xmm0vb5vn}
provides a ``natural'' one-to-one correspondence between classical and nonlocal problems
(the nonlocal case being ``more abundant'', formula~\eqref{0wqurj3o-4uy4ijp30to5Xmm0vb5vn}
allows for a convenient choice).\medskip

The construction in Theorem~\ref{THMFROMATOM}, together with
the estimates for the solutions found through Theorem~\ref{main2} and the convergence result in Proposition~\ref{Proplimito1},
allows one
to recover classical results
established in~\cite[Theorem~2.1]{MR4145478} for the local problem with a limit procedure, according to the following result:

\begin{theorem}\label{fuweioifhqweio7564}
Let~$A:\Omega\to\R^{n\times n}$. Suppose that~$A(x)$ is a symmetric matrix
and denote by~$\lambda_i(x)$ its eigenvalues.

Assume that there exist~$\undertilde\lambda$, $\widetilde\lambda>0$ such that, for any~$i\in\{1,\dots, n\}$ and for any~$x\in\Omega$, 
\begin{equation*}
\undertilde\lambda\le\lambda_i(x)\le \widetilde\lambda.
\end{equation*}

Let~$h:\R\to\R$ satisfy~\eqref{ginftydef}. 
Let~$a$, $f\in L^1(\Omega)$ be such that~\eqref{GQcond} is satisfied.

Then, there exists a unique weak solution~$u_1\in H_0^1(\Omega)\cap L^\infty(\Omega)$ of~\eqref{problemalocale}
satisfying
\begin{eqnarray*}&& \|u_1\|_{L^\infty(\Omega)}\le h^{-1}(Q)\\
{\mbox{and }} &&\|\nabla u_1\|_{L^2(\Omega)}\le C\,h^{-1}(Q) \,\big(\|f\|_{L^1(\Omega)} + Q \|a\|_{L^1(\Omega)}\big)
.\end{eqnarray*}

Moreover, there 
exists a sequence~$M_\ell:\R^n\times\R^n\to\R^{n\times n}$,
satisfying~\eqref{matdefpos} and~\eqref{structuralM} (for any $\varrho\in (0, +\infty]$),
and~$s_j\in(0,1)$ with~$s_j\nearrow1$ as~$j\to+\infty$,
such that, for every~$\varrho\in(0,+\infty]$ and every~$\ell$, $j\in\N$, there exists~$u_{j,\ell}\in H^{s_j}_0(\Omega)\cap L^\infty(\Omega)$ which is the unique weak solution of
\[
\begin{cases}
\mathcal L_{M_\ell,\varrho, s_j} u + a h(u) = f &\mbox{ in } \Omega,\\
u=0 & \mbox{ in } \R^n\setminus\Omega,
\end{cases}
\]
with
\[
\lim_{\ell\to+\infty}
\lim_{j\to+\infty} u_{j,\ell}(x) = u_1(x).
\]
\end{theorem}

Interestingly, Theorem~\ref{fuweioifhqweio7564} also provides an alternative proof of~\cite[Theorem~2.1]{MR4145478},
since it demonstrates the existence of the classical solution~$u_1$ which is obtained
via a limit procedure as~$s\nearrow1$.\medskip

Our paper is organized as follows. 
In Section~\ref{SEC4} we set up the stage for Theorems~\ref{main2} and~\ref{thmminimo}, which are then proved at the end of the section.
In Section~\ref{sectionlimitelocale},  we show that the solution found through Theorem~\ref{main2} converges, in the limit~$s\nearrow 1$, to a solution of the local problem~\eqref{problemalocale}, thus establishing Theorem~\ref{main3}.

In Section~\ref{SEC3}
we discuss the twofold nature of formula~\eqref{aijdefn}, proving that it can also be applied to construct~$M$ starting from~$A$.
This section contains the proof of Theorem~\ref{THMFROMATOM},
which allows to recover the classical result, as stated in Theorem~\ref{fuweioifhqweio7564}.

The paper ends with some appendices that contain the statements and proofs
of some technical results.

\section{Existence theory and proof of Theorems~\ref{main2} and~\ref{thmminimo}}\label{SEC4}

\subsection{Functional analysis preliminaries}\label{SEC4iueryihglsdg}
In this section we introduce the variational framework needed to deal with problems~\eqref{davidnonlocal2} and~\eqref{problemaGxuu} and we provide some preliminary results which will be required to prove Theorems~\ref{main2} and~\ref{thmminimo}.

For~$s\in(0,1)$, we let
\begin{equation}\label{gagliardoseminorm}
[u]_{s}: = \left({c_{n,s}}\iint_{\R^n\times\R^n} \frac{|u(x) - u(y)|^2}{|x - y|^{n+2s}}\, dx\, dy\right)^{\frac12}
\end{equation}
be the Sobolev-Gagliardo seminorm of a measurable function~$u : \R^n \to \R$, see e.g.~\cite{MR2944369}.

The positive normalizing constant~$c_{n,s}$ is chosen in such a way that it provides consistent limits as~$s\searrow 0$ and~$s\nearrow 1$, that is
\[
\lim_{s\nearrow1}[u]_{s}=[u]_1:=\|\nabla u\|_{L^2(\R^n)}\qquad{\mbox{and}}\qquad\lim_{s\searrow0}[u]_s=[u]_0:=\|u\|_{L^2(\R^n)}.
\]
As customary, we define the fractional Sobolev space
\[
H^s(\R^n):=\left\{ u\in L^2(\R^n) : [u]_s <+\infty \right\}
\]
endowed with the norm
\[
\|u\|_{H^s(\R^n)}:= \left([u]^2_s + \|u\|^2_{L^2(\R^n)}\right)^{\frac12}.
\]
Moreover, given an open bounded subset of~$\R^n$ with Lipschitz boundary, we set
\[
H^s_0(\Omega):=\{u\in H^s(\R^n) : u=0 \mbox{ in } \R^n\setminus\Omega\}
\]
endowed with the equivalent norm~$\|\cdot\|_{H^s_0(\Omega)} = [\cdot]_s$.

It is also useful to observe that the norm induced by the kernel~$K$ is equivalent
to the fractional Sobolev-Gagliardo norm, as stated in the following result:

\begin{lemma}\label{LEMMAANCVARRO}
There exists~$C_\star>0$, depending only on~$n$, $s$, $c$, $\varrho$ and~$\Omega$, such that, for all~$u:\R^n\to\R$ with~$u=0$ in~$\R^n\setminus\Omega$,
$$\iint_{\R^n\times\R^n} \frac{|u(x)-u(z)|^2}{|x-z|^{n+2s}}\,dx\,dz
\le C_\star\,\iint_{\R^n\times\R^n}|u(x)-u(z)|^2 K(x,z)\,dx\,dz.$$
\end{lemma}

For convenience of the reader, the proof of this result is presented in Appendix~\ref{APPE1}.

With the setting introduced above, the definition of weak solution of problem~\eqref{davidnonlocal2} reads as follows:

\begin{definition}\label{defnweaksol}
We say that~$u\in H^s_0(\Omega)$ is a weak solution of~\eqref{davidnonlocal2} if, for any~$v\in H^s_0(\Omega)\cap L^\infty(\Omega)$,
\[
\frac{1}{2}\iint_{\R^n\times\R^n} \big(u(x)-u(z)\big)\big(v(x)-v(z)\big)K(x,z)\, dx \,dz + \int_\Omega a(x) h(u(x)) v(x) \, dx = \int_\Omega f(x) v(x)\, dx.
\]
\end{definition}

We now let~$k\ge 0$ and we consider the cutoff function~$G_k:\R\to\R$ defined as
\begin{equation}\label{Gkdefn}
G_k (t):= \begin{cases}
0 &\mbox{ if } |t|\le k,\\
t-k &\mbox{ if } t> k,\\
t+k &\mbox{ if } t<- k.
\end{cases}
\end{equation}

The function~$G_k$ enjoys some useful properties, that we now recall, for the facility of the reader.

\begin{proposition}
Let~$k\ge 0$ and~$G_k:\R\to\R$ be defined as in~\eqref{Gkdefn}. Then, for any~$t$, $r\in\R$, we have that
\begin{align}\label{gprop1}
& t\, G_k (t) = |t| |G_k (t)|,\\ \label{gprop3}
& {\mbox{$G_k (t)$ is nondecreasing,}}\\ \label{gprop2}
\mbox{ and }\quad &|G_k (t) -  G_k(r)| \le |t-r|.
\end{align}
\end{proposition}

\begin{proof}
We can write~\eqref{Gkdefn} in a compact way as
\[
G_k(t)= \int_0^t \chi_{\R\setminus [-k, k]} (x) \,dx.
\]
{F}rom this, the property in~\eqref{gprop1} plainly follows.

Also, by inspection of all possible cases, one can check that~\eqref{gprop3} holds true. 

Moreover,
\[
|G_k (t)-G_k(r)| = \left|\int_{r}^t \chi_{\R\setminus [-k, k]} (x) \,dx \right| \le |t-r|,
\]
which proves~\eqref{gprop2}. 
\end{proof}

\begin{proposition}\label{Gknellospazio}
Let~$k\ge 0$ and~$u\in H^s_0(\Omega)$. Then, $G_k (u)$ belongs to~$H^s_0(\Omega)$.
\end{proposition}

\begin{proof}
{F}rom~\eqref{gprop2}, we have that~$|G_k(u)|\le |u|$, and therefore~$G_k(u)\in L^2(\R^n)$.

Moreover, by~\eqref{gagliardoseminorm} and~\eqref{gprop2}, we infer that
\[
\begin{split}&
[G_k (u)]^2_s = c_{n, s} \iint_{\R^n\times\R^n} \frac{|G_k(u(x))- G_k(u(y))|^2}{|x-y|^{n+2s}}\, dx \,dy\\
&\qquad\le c_{n, s}\iint_{\R^n\times\R^n} \frac{|u(x) - u(y)|^2}{|x-y|^{n+2s}}\, dx\, dy= [u]^2_s,
\end{split}
\]
and therefore~$G_k(u)\in H^s(\R^n)$.

Furthermore, in~$\R^n\setminus\Omega$ we have that~$u=0$, thus~$G_k (u)=0$. Gathering these
pieces of information, we obtain the desired claim.
\end{proof}

\subsection{Proof of Theorem~\ref{thmminimo}}
The proof of Theorem~\ref{thmminimo} is deduced from a minimization argument, whose details are as follows.

We set
\begin{equation*}
H(t):=\int_0^t h(\tau) \,d\tau.
\end{equation*}

We consider the functional~$J:H^s_0(\Omega)\to\R\cup\{+\infty,-\infty\}$ defined as
\begin{equation}\label{functionalGxu}
J(u):=\begin{cases}
\displaystyle\frac{1}{4}\displaystyle\iint_{\R^n\times\R^n} |u(x)-u(z)|^2K(x,z)\, dx\, dz +\int_{\Omega} \eta(x)H(u(x))\, dx-\int_\Omega \zeta(x) u(x)\, dx \\
\qquad \qquad\qquad\mbox{ if } \eta (H\circ u)\in L^1(\Omega),\\
+\infty \qquad \mbox{ if } \eta (H\circ u)\notin L^1(\Omega).
\end{cases}
\end{equation}
We now check the following properties of~$J$:

\begin{lemma}\label{q0wdfujgobn0ytuik}
Let~$u\in H^s_0(\Omega)$. If~$\eta(H\circ u)\in L^1(\Omega)$, then~$|J(u)|<+\infty$.

In particular, $J:H^s_0(\Omega)\to\R\cup\{+\infty\}$.\end{lemma}

\begin{proof}
If~$u\in H^s_0(\Omega)$, then from the upper bound in~\eqref{QrehtgDWFfbg0jolwfdv:Qwfgb} we have that
\[
\iint_{\R^n\times\R^n} |u(x)-u(z)|^2K(x,z)\, dx\, dz
\le C\iint_{\R^n\times\R^n} \frac{|u(x)-u(z)|^2}{|x-z|^{n+2s}}\, dx\, dz=
\frac{C}{c_{n,s}}[u]^2_{s}<+\infty.
\]
Moreover,
\[
\left|\int_{\Omega} \eta(x)H(u(x)) \,dx\right|\le  \int_\Omega \eta(x)\,H(u(x))\,dx<+\infty
\]
and, since~$u\in L^2(\Omega)$, by the H\"older inequality,
\begin{equation}\label{qwd0fojgbnlgX5ryku}
\left|\int_\Omega \zeta(x) u(x)\, dx\right|\le \|\zeta\|_{L^2(\Omega)} \|u\|_{L^2(\Omega)} <+\infty.
\end{equation}
Combining these considerations with~\eqref{functionalGxu}, we get desired result.
\end{proof}

We also observe that
\begin{equation}\label{ufeiwqgfwr8743y8hgh}
{\mbox{critical points of the functional~$J$ are weak solutions of~\eqref{problemaGxuu}.}}
\end{equation}

\begin{lemma}\label{Lemmawlsc}
The functional~$J$ is weakly lower semicontinuous with respect to the topology of~$H^s_0(\Omega)$,
namely for every sequence~$u_k\in H^s_0(\Omega)$ which converges weakly to some~$u\in H^s_0(\Omega)$ as~$k\to+\infty$ we have that
\begin{equation*}
\liminf_{k\to +\infty}J(u_k)\ge J(u).
\end{equation*}
\end{lemma}

\begin{proof} To prove the desired result, we argue by contradiction and we suppose that
there exists a sequence~$u_k\in H^s_0(\Omega)$ weakly converging to some~$u\in H^s_0(\Omega)$ as~$k\to+\infty$ and such that
\begin{equation}\label{qwd0fojgbnlgX5ryku.2}
\lim_{k\to +\infty}J(u_k)< J(u).
\end{equation}
In particular, 
\begin{equation}\label{0qwdfojghb0rntijkXk} \lim_{k\to +\infty}J(u_k)<+\infty,\end{equation}
which, in tandem with Lemma~\ref{q0wdfujgobn0ytuik},
implies that~$\eta(H\circ u_k)\in L^1(\Omega)$ for~$k$ sufficiently large.

By the compactness result in~\cite[Theorem~7.1]{MR2944369} we have that
there exists a subsequence~$u_{k_j}$ strongly
converging to~$u$ in~$L^2(\Omega)$ and a.e. in~$\Omega$. Therefore,
by Fatou's Lemma,
\begin{equation*}
\liminf_{j\to +\infty}\iint_{\R^n\times\R^n} |u_{k_j}(x)-u_{k_j}(z)|^2K(x,z)\, dx\, dz \ge\iint_{\R^n\times\R^n} |u(x)-u(z)|^2K(x,z) \,dx\, dz.
\end{equation*}

By~\eqref{ginftydef}, it follows that
\begin{equation}\label{Gge0}
H(t)\ge 0\quad\mbox{ for any } t\in\R.
\end{equation}
Consequently, since~$\eta\ge 0$ and~$H$ is continuous, it follows from
Fatou's Lemma that
\begin{equation*}
\liminf_{j\to +\infty}\int_\Omega \eta(x)H(u_{k_j}(x))\,dx\ge \int_\Omega \eta(x)H(u(x))\,dx.
\end{equation*}

Furthermore, since~$\zeta\in L^2(\Omega)$, the strong convergence of~$u_{k_j}$ in~$L^2(\Omega)$ implies that
\begin{equation*}
\lim_{j\to +\infty} \int_\Omega \zeta(x) u_{k_{j}}(x)\, dx = \int_\Omega \zeta(x) u(x)\, dx.
\end{equation*}

Gathering these pieces of information and recalling the definition of~$J$ in~\eqref{functionalGxu}, 
we conclude that 
\begin{equation}\label{PADLSFNVBFKOLKNJDFDFO9diwfkf}\begin{split}
&\lim_{j\to +\infty}J(u_{k_{j}})\\&
\ge \liminf_{j\to +\infty}
\displaystyle\frac{1}{4}\displaystyle\iint_{\R^n\times\R^n} |u_{k_{j}}(x)-u_{k_{j}}(z)|^2K(x,z)\, dx\, dz +\int_{\Omega} \eta(x)H(u_{k_{j}}(x))\, dx-\int_\Omega \zeta(x) u_{k_j}(x)\, dx\\&\ge
\displaystyle\frac{1}{4}\displaystyle\iint_{\R^n\times\R^n} |u(x)-u(z)|^2K(x,z)\, dx\, dz +\int_{\Omega} \eta(x)H(u(x))\, dx-\int_\Omega \zeta(x) u(x)\, dx
.\end{split}\end{equation}

Hence, recalling~\eqref{qwd0fojgbnlgX5ryku},
$$ \int_{\Omega} \eta(x)H(u(x))\, dx\le\int_\Omega \zeta(x) u(x)\, dx+\lim_{j\to +\infty}J(u_{k_{j}})\le \|\zeta\|_{L^2(\Omega)} \|u\|_{L^2(\Omega)} +\lim_{j\to +\infty}J(u_{k_{j}}).$$
We stress that the latter quantity is finite, owing to~\eqref{0qwdfojghb0rntijkXk}, and therefore~$\eta(H\circ u)\in L^1(\Omega)$.

As a result, by Lemma~\ref{q0wdfujgobn0ytuik},
$$ J(u)=\displaystyle\frac{1}{4}\displaystyle\iint_{\R^n\times\R^n} |u(x)-u(z)|^2K(x,z)\, dx\, dz +\int_{\Omega} \eta(x)H(u(x))\, dx-\int_\Omega \zeta(x) u(x)\, dx.$$

{F}rom this, \eqref{qwd0fojgbnlgX5ryku.2} and~\eqref{PADLSFNVBFKOLKNJDFDFO9diwfkf} we obtain that
\begin{eqnarray*}&& J(u)>\lim_{j\to +\infty}J(u_{k_{j}})\\&&\quad\ge
\displaystyle\frac{1}{4}\displaystyle\iint_{\R^n\times\R^n} |u(x)-u(z)|^2K(x,z)\, dx\, dz +\int_{\Omega} \eta(x)H(u(x))\, dx-\int_\Omega \zeta(x) u(x)\, dx=J(u).\end{eqnarray*}
This is a contradiction establishing the desired result.
\end{proof}

\begin{lemma}\label{lemmacoercivo}
The functional~$J$ is coercive, namely
\begin{equation*}
\lim_{[u]_s\to+\infty}\frac{J(u)}{[u]_s}=+\infty.
\end{equation*}
\end{lemma}

\begin{proof}
We can suppose that~$H\circ u\in L^1(\Omega)$. 
{F}rom Lemma~\ref{LEMMAANCVARRO} and~\eqref{Gge0},
\begin{equation}\label{dtfyguhinjk}
\begin{split}
J(u)&=\frac{1}{4}\iint_{\R^n\times\R^n} |u(x)-u(z)|^2K(x,z)\, dx\, dz +\int_{\Omega} \eta(x)H(u(x))\, dx-\int_\Omega \zeta(x) u(x)\, dx\\
&\ge \frac{\widetilde{C}}{4}[u]_s^2
+ \int_{\Omega} \eta(x)H(u)\, dx-\int_\Omega \zeta(x) u(x)\, dx\\
&\ge \frac{\tilde{C}}{4}[u]_s^2-\int_\Omega |\zeta(x)| |u(x)|\, dx,
\end{split}
\end{equation} for some~$\widetilde{C}>0$.

Now, by the H\"older inequality and~\cite[Theorem~6.7]{MR2944369}, there exists a constant~$\widehat{C}>0$, depending only on~$n$, $s$ and~$\Omega$, such that
\[
\int_\Omega |\zeta(x)| |u(x)|\, dx\le \|\zeta\|_{L^2(\Omega)} \|u\|_{L^2(\Omega)}\le \widehat{C}\|\zeta\|_{L^2(\Omega)} [u]_s.
\]
{F}rom this and~\eqref{dtfyguhinjk}, we deduce that
\[
J(u)\ge \frac{\widetilde{C}}{4}[u]_s^2 -\widehat{C}\|\zeta\|_{L^2(\Omega)} [u]_s.
\]
Hence, 
\[
\lim_{[u]_s\to+\infty}\frac{J(u)}{[u]_s}\ge \lim_{[u]_s\to+\infty}
\left(\frac{\widetilde{C}}{4}[u]_s -\widehat{C}\|\zeta\|_{L^2(\Omega)}\right)= +\infty,
\] as desired.
\end{proof}

With the work done so far, the proof of Theorem~\ref{thmminimo} can be completed as follows.

\begin{proof}[Proof of Theorem~\ref{thmminimo}]
Since~$J$ is weakly lower semicontinuous (thanks to Lemma~\ref{Lemmawlsc})
and
coercive (thanks to Lemma~\ref{lemmacoercivo})
and the space~$H^s_0(\Omega)$ is reflexive, by~\cite[Theorem~1.2]{MR2431434}, there exists a minimizer for the functional~$J$,
which is a weak solution of problem~\eqref{problemaGxuu} (recall~\eqref{ufeiwqgfwr8743y8hgh}).
\end{proof}

\subsection{Proof of Theorem~\ref{main2}}

We are now in the position to prove Theorem~\ref{main2}.

\begin{proof}[Proof of Theorem~\ref{main2}]
For any~$j\in\N$, we set
\begin{equation}\label{fnandefn}
f_j(x):= \dfrac{f(x)}{1+\frac{|f(x)|}{j}}\quad \mbox{ and }\quad a_j(x):= \dfrac{a(x)}{1+\frac{Q\, a(x)}{j}}.
\end{equation}
We observe that~$f_j$, $a_j \in L^\infty(\Omega)$ for any~$j\in\N$ and 
\begin{equation}\label{fnanlen}
\|f_j\|_{L^\infty(\Omega)}\le j \quad\mbox{ and }\quad \|a_j\|_{L^\infty(\Omega)}\le \frac{j}{Q}.
\end{equation}

In addition, setting
$$ \psi_j(t):=\frac{jt}{j+t},$$
one can check that~$\psi_j$ is nondecreasing, and therefore from~\eqref{GQcond} we deduce that
\begin{equation}\label{modulofn}
|f_j(x)| = \dfrac{|f(x)|}{1+\frac{|f(x)|}{j}} =\psi_j(|f(x)|)\le \psi_j(Qa(x))
=\dfrac{Q a(x)}{1+\frac{Q a(x)}{j}} = Q a_j(x).
\end{equation}

We now consider the approximated problems
\begin{equation}\label{Gprobapprox}
\begin{cases}
\mathcal L_{K} u + a_j \, h(u) = f_j &\mbox{ in } \Omega,\\
u=0 &\mbox{ in } \R^n\setminus\Omega,
\end{cases}
\end{equation}
and we observe that we are in the setting of Theorem~\ref{thmminimo} with~$\eta:=a_j$ and~$\zeta:= f_j$.
Consequently, for any~$j\in\N$, there exists a weak solution~$u_j$ of~\eqref{Gprobapprox},
that is, for any~$j\in\N$, there exists~$u_j\in H^s_0(\Omega)$ such that, for any~$v\in H^s_0(\Omega)$, 
\begin{equation}\label{Gweakn}
\begin{split}
&\frac12\iint_{\R^n\times\R^n} \big(u_j(x)-u_j(z)\big)\big(v(x)-v(z)\big)K(x,z)\, dx\, dz +\int_\Omega a_j(x) h(u_j(x))v(x)\, dx\\
&\qquad = \int_\Omega f_j(x) v(x)\, dx.
\end{split}
\end{equation}
 
Now, in light of the assumptions on~$h$ in~\eqref{ginftydef}, we have that the~$h$ is invertible in~$(-\gamma, \gamma)$. Hence, denoted by~$h^{-1}$ this inverse and recalling that~$Q\in(0,\gamma)$ (see~\eqref{GQcond}), we have
\begin{equation*}
k:= h^{-1}(Q)>0.
\end{equation*}
Recalling the definition in~\eqref{Gkdefn} and exploiting the assumptions on~$h$ in~\eqref{ginftydef}, 
\begin{equation}\label{GkunQ}
G_k (u_j):= \begin{cases}
0 &\mbox{ if } |h(u_j)|\le Q,\\
u_j-k &\mbox{ if } h(u_j)> Q,\\
u_j+k &\mbox{ if } h(u_j)<- Q.
\end{cases}
\end{equation}

Moreover, by Proposition~\ref{Gknellospazio}, we can use~$G_k (u_j)$ as a test function in~\eqref{Gweakn},
to obtain that
\begin{equation*}
\begin{split}
&\frac12\iint_{\R^n\times\R^n} \big(u_j(x)-u_j(z)\big)\big(G_k (u_j(x))-G_k (u_j(z))\big)K(x,z)\, dx\, dz\\&\qquad +\int_\Omega a_j(x) h(u_j(x))G_k (u_j(x))\, dx = \int_\Omega f_j(x) G_k (u_j(x))\, dx.
\end{split}
\end{equation*}
We also observe that, thanks to~\eqref{ginftydef} and~\eqref{gprop1}, 
\[
h(t) G_k(t) = |h(t) G_k(t)| \quad\mbox{ for any } t\in\R,
\]
and therefore
\begin{equation*}
\begin{split}
&\frac12\iint_{\R^n\times\R^n} \big(u_j(x)-u_j(z)\big)\big(G_k (u_j(x))-G_k (u_j(z))\big)K(x,z)\, dx\, dz \\&\qquad
+\int_\Omega a_j(x) |h(u_j(x))|\,|G_k (u_j(x))|\, dx = \int_\Omega f_j(x) G_k (u_j(x))\, dx.
\end{split}
\end{equation*}
{F}rom this and~\eqref{modulofn}, we thus deduce that
\begin{equation}\label{Gsxdcfvjnk}
\begin{split}
&\frac12\iint_{\R^n\times\R^n} \big(u_j(x)-u_j(z)\big)\big(G_k (u_j(x))-G_k (u_j(z))\big)K(x,z)\, dx\, dz\\
&\qquad +\int_\Omega a_j(x) |h(u_j(x))|\,|G_k (u_j(x))|\, dx \le
Q\int_\Omega a_j(x)\, |G_k (u_j(x))|\, dx.
\end{split}
\end{equation}

Moreover, from~\eqref{gprop3} and~\eqref{gprop2} we infer that
\begin{equation*}
\begin{split}
&\iint_{\R^n\times\R^n} |G_k (u_j(x)) -  G_k(u_j(z))|^2K(x,z)\, dx\, dz \\
\qquad&\le \iint_{\R^n\times\R^n} \big|G_k (u_j(x)) -  G_k(u_j(z))\big| \big|u_j(x)-u_j(z)\big|K(x,z)\, dx\, dz \\
\qquad&= \iint_{\R^n\times\R^n} \big(G_k (u_j(x)) -  G_k(u_j(z))\big) \big(u_j(x)-u_j(z)\big)K(x,z)\, dx\, dz.
\end{split}
\end{equation*}
Combining this with~\eqref{Gsxdcfvjnk}, we get that
\begin{eqnarray*}&&
\frac12\iint_{\R^n\times\R^n} \big|G_k (u_j(x))- G_k (u_j(z) )\big|^2K(x,z)\, dx\, dz\\&&\qquad
+\int_\Omega a_j(x) \big(|h(u_j(x))|-Q\big) |G_k (u_j(x))|\, dx \le 0.
\end{eqnarray*}
Thus, since~$G_k (u_j(x))=0$ if~$|h(u_j(x))|\le Q$
(recall the definition in~\eqref{GkunQ}), we conclude that
\begin{eqnarray*}&&
\frac12\iint_{\R^n\times\R^n} \big|G_k (u_j(x))- G_k (u_j(z) )\big|^2K(x,z)\, dx\, dz\\&&\qquad
+\int_{\Omega\cap \{ |h(u_j(x))|>Q\}} a_j(x) \big(|h(u_j(x))|-Q\big) |G_k (u_j(x))|\, dx \le 0.
\end{eqnarray*}
This and Lemma \ref{LEMMAANCVARRO} entail that~$G_k (u_j(x))=0$ for a.e.~$x\in\Omega$, and therefore
\begin{equation}\label{GleQ}
\|u_j\|_{L^\infty(\Omega)}\le h^{-1}(Q).
\end{equation}

We now claim that 
\begin{equation}\label{claimunbddHs}
\mbox{$u_j$ is bounded in~$H^s_0(\Omega)$}.
\end{equation}
Indeed, by~\eqref{fnandefn} and~\eqref{GleQ}, and being~$h$ strictly increasing
due to~\eqref{ginftydef},  
\begin{equation*}
\begin{split}
&|a_j(x) h(u_j(x)) u_j(x)|\le a(x)h(\|u_j\|_{L^\infty(\Omega)}) \|u_j\|_{L^\infty(\Omega)} \le a(x)\, Q\, h^{-1}(Q)\\
\mbox{and }\quad &|f_j(x) u_j(x)|\le |f(x)|\,h^{-1}(Q).
\end{split}
\end{equation*}
Thus, by using~$u_j\in H^s_0(\Omega)$ as a test function in~\eqref{Gweakn}, we find that
\begin{equation}\label{i06uyretfasjhdfqwtfetq}
\begin{split}&
\frac12 \iint_{\R^n\times\R^n} \big(u_j(x)-u_j(z)\big)^2
K(x,z)\, dx\, dz= \int_\Omega f_j(x) u_j(x)\,dx - \int_\Omega a_j(x) h\big(u_j(x)\big) u_j(x) \,dx\\
& \qquad\le h^{-1}(Q) \Big(\|f\|_{L^1(\Omega)} + Q \|a\|_{L^1(\Omega)}\Big).
\end{split}
\end{equation}
This and Lemma~\ref{LEMMAANCVARRO} entail~\eqref{claimunbddHs}.

By~\eqref{claimunbddHs}, there exists~$u\in H^s_0(\Omega)$ such that, up to subsequences, 
\begin{align}\label{weakconvOmega}
&u_j\rightharpoonup u \quad\mbox{ in } H^s_0(\Omega),\\ \label{aeinOmega}
{\mbox{and }}\quad &u_j\to u \quad\mbox{ a.e.  in } \Omega.
\end{align}
Moreover,  by~\eqref{GleQ},  
\begin{equation}\label{GuleQ}
\|u\|_{L^\infty(\Omega)}\le h^{-1}(Q).
\end{equation}
Also, by~\eqref{i06uyretfasjhdfqwtfetq},
\eqref{aeinOmega} and Fatou's Lemma, we infer that 
\begin{equation}\label{GuleQBIS}\begin{split}
\iint_{\R^n\times\R^n} |u(x)-u(z)|^2K(x,z)\,dx\, dz &\le \liminf_{j\to +\infty}\iint_{\R^n\times\R^n} |u_j(x)-u_j(y)|^2K(x,z) \,dx\, dz\\&\le 
2 h^{-1}(Q)\Big(\|f\|_{L^1(\Omega)} + Q \|a\|_{L^1(\Omega)}\Big).
\end{split}\end{equation}

Now, we check that~$u$ is a weak solution of problem~\eqref{davidnonlocal2}
according to Definition~\ref{defnweaksol}. For this, we take~$v\in H^s_0(\Omega)\cap L^\infty(\Omega)$ as a
test function in~\eqref{Gweakn} and we obtain that
\begin{equation}\label{Gweakn2}
\begin{split}
&\iint_{\R^n\times\R^n} \big(u_j(x)-u_j(z)\big)\big(v(x)-v(z)\big)K(x,z)\, dx \,dz +\int_\Omega a_j(x) h(u_j(x))v(x)\, dx\\
&\quad= \int_\Omega f_j(x) v(x)\, dx.
\end{split}
\end{equation}
The aim is now to pass to the limit as~$j\to+\infty$ the equation in~\eqref{Gweakn2}.
For this, we use~\eqref{GleQ}  and the properties of~$h$ in~\eqref{ginftydef}
to see that
\begin{align*}
&|a_j h(u_j) v|\le h(h^{-1}(Q)) a \|v\|_{L^\infty(\Omega)} =Q a \|v\|_{L^\infty(\Omega)} \in L^1(\Omega)\\
\mbox{ and }\quad & |f_j v | \le  f \|v\|_{L^\infty(\Omega)} \in L^1(\Omega).
\end{align*}
Hence, we employ the Dominated Convergence Theorem and 
the weak convergence in~\eqref{weakconvOmega} to pass to the limit as $j\to +\infty$ in~\eqref{Gweakn2} and deduce that $u$ satisfies
\begin{equation*}
\iint_{\R^n\times\R^n} \big(u(x)-u(z)\big)\big(v(x)-v(z)\big)K(x,z)\, dx \,dz +\int_\Omega a(x) h(u(x))v(x)\, dx= \int_\Omega f(x) v(x)\, dx,
\end{equation*} as desired.

Furthermore, the estimates in~\eqref{GLinftyu} and~\eqref{seminormaus}
follow from~\eqref{GuleQ} and~\eqref{GuleQBIS}.

We now show that the solution that we found is unique. To this aim, we take~$u_1$,  $u_2\in H^s_0(\Omega)\cap L^\infty(\Omega)$ solutions of~\eqref{davidnonlocal2}. We point out that~$u_1-u_2$ can be used
as a test function in the weak formulation of~\eqref{davidnonlocal2}. In this way,
one obtains that
\begin{eqnarray*}&&
\frac{1}{2}\iint_{\R^n\times\R^n} \big((u_1-u_2)(x)-(u_1-u_2)(z)\big)^2K(x,z)\, dx \,dz \\&&\qquad
+ \int_\Omega a(x) \big(h(u_1(x))-h(u_2(x))\big) (u_1(x)-u_2(x)) \, dx = 0.
\end{eqnarray*}
{F}rom the fact that~$a\ge0$, the assumptions on~$h$ in~\eqref{ginftydef} and Lemma \ref{LEMMAANCVARRO}, we deduce that~$u_1 =u_2$ a.e. in~$\Omega$.
\end{proof}

\begin{remark}
It is tempting to use the Schauder Fixed Point Theorem to prove that problem \eqref{problemaGxuu} admits a weak solution.
While this approach has been previously used in \cite{MR3304596, MR3760750, MR4145478}, some care is needed to make the full argument work.
Indeed, given $f\in L^2(\Omega)$ and $s\in(0,1]$, one can define $S:L^2(\Omega)\to H_0^s(\Omega)$ as the solution operator $u=S(f)$ of the problem
\begin{equation}\label{apppro1}
\begin{cases}
(-\Delta)^s u = f &\mbox{ in } \Omega\\
u=0 &\mbox{ in } \R^N\setminus\Omega.
\end{cases}
\end{equation}
By either classical or fractional regularity theory, it follows that $S$ is linear and continuous.

Then, one could be tempted to introduce the operator $T:H_0^s(\Omega)\to H_0^s(\Omega)$ defined as
\begin{equation}\label{Tdefn}
T(u):= S(\zeta-\eta\,h(u)).
\end{equation}
With this choice, $T(u)$ would be the solution of the problem \eqref{apppro1} with $f:=\zeta-\eta\,h(u)$.

At a first glance, this seems a convenient setting for the application of the Schauder Fixed Point Theorem, up to showing that
$T$ is continuous, compact and leaves a suitable closed ball invariant.

However,  even if $\eta\in L^\infty(\Omega)$ and $\zeta\in L^2(\Omega)$, the assumptions on the function $h$ are, in general, not strong enough to guarantee that $\zeta-\eta\,h(u)\in L^2(\Omega)$
(the issues arising when $h$ grows faster than linear at infinity: 
for example, if $\zeta:=0$, $\eta:=1$, $h(t):=t^3$, $u(x):= |x|^{-\frac13}$ and $\Omega\ni0$, then $\zeta-\eta\,h(u)
\not\in L^2(\Omega)$). This caveat prevents one from defining the operator $T$ and makes the Schauder Fixed Point Theorem unavailable in this situation (likely, a two-step approach would be needed in cases like this, based first on a truncation of the function $h$ with a
strictly increasing and bounded function and on
a subsequent limit argument).
\end{remark}

\section{Asymptotic behavior of nonlocal solutions and proof of Theorem~\ref{main3}}\label{sectionlimitelocale}

In this section we show that the unique solution of problem~\eqref{davidnonlocal2MS}, obtained through Theorem~\ref{main2},
converges as~$s\nearrow1$ to a function which is
the unique solution of problem~\eqref{problemalocale}.
The existence of this classical solution was obtained in~\cite[Theorem~2.1]{MR4145478}, but
our proof does not need to assume this result and can be also seen as a different approach to the classical
theory, which is now recovered as a byproduct of a limit case.
In addition, the stability of the solutions of Theorem~\ref{main2} as~$s\nearrow1$
is interesting in itself, since it relies on a series of uniform estimates which do not come
from the standard regularity theory.

For the sake of convenience,  for any~$s\in (0,1)$, we denote by~$u_s\in H^s_0(\Omega)\cap L^\infty(\Omega)$ the unique weak solution of problem~\eqref{davidnonlocal2MS} and by~$u_1\in H_0^1(\Omega)\cap L^\infty(\Omega)$ the unique weak solution of problem~\eqref{problemalocale}.

Now we show that, given~$u\in H^s_0(\Omega)$, we can decrease the seminorm of~$u$ by performing a mollification procedure.

\begin{lemma}\label{lemmasemimoll}
Let~$s\in (0, 1)$ and~$u\in H^s_0(\Omega)$. Let~$\eta\in C^\infty_c(B_1)$ with~$\int_{\R^n} \eta(x) \,dx =1$. 

For any~$\varepsilon>0$, let
\begin{equation}\label{mollifier1}
\eta_\varepsilon (x):=\varepsilon^{-n} \eta\left(\frac{x}{\varepsilon}\right)\qquad{\mbox{and }}\qquad u_\varepsilon(x):= (\eta_\varepsilon\ast u)(x).
\end{equation}

Then, 
\[
[u_\varepsilon]_s\le [u]_s.
\]
\end{lemma}

\begin{proof}
By definition, we have that~$\eta_\varepsilon\in C^\infty(\Omega)$, $\mbox{supp}(\eta_\varepsilon)\subset B_\varepsilon$ and 
\begin{equation*}
\int_{\R^n} \eta_\varepsilon(x) \,dx =1.
\end{equation*}
Hence, recalling the definition in~\eqref{gagliardoseminorm} and using the Jensen inequality, we get
\[
\begin{split}&
[u_\varepsilon]^2_s = \iint_{\R^n\times\R^n} \frac{|u_\varepsilon(x)-u_\varepsilon(y)|^2}{|x-y|^{n+2s}} \,dx \, dy\le \iint_{\R^n\times\R^n}\left(\,\int_{\R^n}\frac{|u(x-z)-u(y-z)|^2}{|x-y|^{n+2s}} \eta_\varepsilon(z) \,dz\right)\,dx\, dy\\
&\qquad\quad=\iint_{\R^n\times\R^n} \frac{|u(x)-u(y)|^2}{|x-y|^{n+2s}} \,dx\, dy \left(\,\int_{\R^n} \eta_\varepsilon(z)\, dz\right) = [u]^2_s,
\end{split}
\]
as desired.
\end{proof}

We are now in the position to prove Theorem~\ref{main3}.

\begin{proof}[Proof of Theorem~\ref{main3}]
We will show that
every sequence of solutions of the nonlocal problem~\eqref{davidnonlocal2MS} is precompact and possesses a subsequence
converging to the unique solution of
the classical problem~\eqref{problemalocale}, which in turn implies 
that all sequences actually converge to the same classical solution, thus yielding the desired result.

We consider a sequence of weak solutions $u_{s_j}$
of~\eqref{davidnonlocal2MS} with~$s_j$ approaching~$1$ as~$j\to+\infty$.
Without loss of generality, we can assume that~$s_j\in (1/2, 1)$.

By~\eqref{GLinftyu} and~\eqref{seminormaus}, we have that
\begin{equation}\label{oufewyfi6658oyewoiygoihy}
\|u_{s_j}\|_{L^\infty(\Omega)}\le h^{-1}(Q)\end{equation}
and \begin{equation*}
c_{n,s_j}\iint_{\R^n\times B_\varrho(x)} \frac{\big(u(x) - u(z)\big)^2}{|M(z,x-z)(x-z)|^{n+2s_j}}\,dx\, dz
\le 2h^{-1}(Q) \,\big(\|f\|_{L^1(\Omega)} + Q \|a\|_{L^1(\Omega)}\big).
\end{equation*}
{F}rom this, the assumption~\eqref{matdefpos} and Lemma~\ref{LEMMAANCVARRO}, we also see that
\begin{equation*}
[u_{s_j}]_{s_j}^2
\le 2C_{\alpha,\beta,n} h^{-1}(Q) \,\big(\|f\|_{L^1(\Omega)} + Q \|a\|_{L^1(\Omega)}\big).
\end{equation*}

We now take~$\overline{s}\in(1/2,1)$. By~\cite[Lemma~2.1]{MR4736013}, we thereby find that there exists a positive constant~$C$, only depending on~$n$ and~$\Omega$, such that, when~${s_j}\in(\overline{s}, 1)$,
\begin{equation}\label{rtcyivuobinjokmpl,}
[u_{s_j}]^2_{\overline s}\le C [u_{s_j}]^2_{s_j} \le 2C\,C_{\alpha,\beta,n} h^{-1}(Q) \,\big(\|f\|_{L^1(\Omega)} + Q \|a\|_{L^1(\Omega)}\big)=:C_1,
\end{equation}
namely the sequence~$u_{s_j}$ is bounded in~$H^{\overline s}_0(\Omega)$, uniformly in~$j$.

Thus, we can extract a subsequence~$s_{j_k}$ for which
there exists~$w\in H_0^{\overline s}(\Omega)$ such that, as~$k\to+\infty$,
\begin{alignat}{2}\nonumber
&u_{s_{j_k}}\rightharpoonup w &&\mbox{ in } H^{\overline s}(\R^n)\\ \label{convL222}
& u_{s_{j_k}} \to w &&\mbox{ in } L^2(\R^n), \\ \label{usaetou}
\mbox{ and }\quad & u_{s_{j_k}} \to w &&\mbox{ a.e. in } \R^n
\end{alignat}
and (see e.g.~\cite[Theorem~4.9]{MR2759829}) there exist~$g\in L^2(\Omega)$ such that
\begin{equation}\label{QUELLA}
|u_{s_{j_{k}}}(x)|\le g(x)\end{equation} for any~$k\in\N$ and for a.e.~$x\in \Omega$.

We will actually show below that~$w$ is the solution~$u_1$ of the classical problem~\eqref{problemalocale}.

By~\eqref{rtcyivuobinjokmpl,}, \eqref{usaetou} and Fatou's Lemma, we conclude that
\begin{equation}\label{ionkmdaslòmdsalk}
[w]^2_{\overline s}\le C_1.
\end{equation}
Moreover, we remark that~\eqref{oufewyfi6658oyewoiygoihy} entails
that
\begin{equation}\label{QUESTA}
\|w\|_{L^\infty(\Omega)}\le h^{-1}(Q).\end{equation}

We now claim that
\begin{equation}\label{claimH1}
\|\nabla w\|_{L^2(\R^n)}\le C_1.
\end{equation}
To prove this, we define
\[
w_{\varepsilon}(x):= (\eta_\varepsilon\ast w)(x),
\]
being~$\eta_\varepsilon$ as in~\eqref{mollifier1}.
By~\cite[Appendix~C4, Theorem~6]{MR1625845}, we have that
\begin{equation}\label{limitepsto0}
\lim_{\varepsilon\searrow 0} w_{\varepsilon} (x)= w(x)\quad{\mbox{ for a.e. }}x\in\R^n.
\end{equation}

Moreover, \eqref{QUESTA} gives that
\begin{equation}\label{hkwdhew8437gfks87654}
\|w_{\varepsilon}\|_{L^\infty(\R^n)}\le \|w\|_{L^\infty(\R^n)}\le h^{-1}(Q).
\end{equation}
Also, in light of Lemma~\ref{lemmasemimoll} and~\eqref{ionkmdaslòmdsalk},
\[
[w_{\varepsilon}]^2_{\overline s} \le [w]^2_{\overline s}\le C_1,
\]
and therefore
\begin{equation}\label{1s3-wrifpg-0}
\lim_{\overline s\nearrow 1}[w_{\varepsilon}]^2_{\overline s}\le \lim_{\overline s\nearrow 1} [w]^2_{\overline s}\le C_1.
\end{equation}

Moreover, since~$w_{\varepsilon}\in C^\infty_c(\R^n)$, we have that
\[
\lim_{\overline s\nearrow 1}[w_{\varepsilon}]^2_{\overline s} = \|\nabla w_{\varepsilon}\|^2_{L^2(\R^n)}.
\]
Combining this and~\eqref{1s3-wrifpg-0}, we deduce that
\[
\|\nabla w_{\varepsilon}\|^2_{L^2(\R^n)}\le C_1.\]

{F}rom this, \eqref{limitepsto0} and~\eqref{hkwdhew8437gfks87654}, for an infinitesimal subsequence~$\varepsilon_\ell$,
\[
w_{\varepsilon_\ell}\rightharpoonup w \mbox{ in } H^1(\R^n).
\]
This shows that~$w\in H^1(\R^n)$ and 
\[
\|\nabla w\|_{L^2(\R^n)}^2 \le \liminf_{\ell\to+\infty} \|\nabla w_{\varepsilon_\ell}\|_{L^2(\R^n)}^2 \le C_1,
\]
which establishes~\eqref{claimH1}.  

From~\eqref{claimH1}, since $w\in H_0^{\overline s}(\Omega)$, we infer that $w\in H^1_0(\Omega)$.

We now show that~$w$ is a weak solution (and, in fact, the unique  solution, due to the comparison principle)
of~\eqref{problemalocale}, namely, for any~$\varphi\in C^\infty_c(\R^n)$, $w$ satisfies
\begin{equation}\label{uastsolution}
\int_\Omega A(x) \nabla w(x)\cdot\nabla \varphi(x) \,dx = \int_\Omega \big(f(x) - a(x) h(w(x))\big)\varphi(x) \,dx.
\end{equation} 
To prove this, we notice that by~\eqref{ginftydef} and~\eqref{GLinftyu}, and since~$a\in L^1(\Omega)$,
\[
|a(x) h(u_{s_{j_k}})\varphi(x)|\le Q\|\varphi\|_{L^\infty(\R^n)} a(x)\in L^1(\Omega).
\]
Hence, from this, \eqref{usaetou} and the Dominated Convergence Theorem, we infer that
\begin{equation}\label{convergenzadx}
\lim_{k\to+\infty}\int_\Omega a(x) h(u_{s_{j_k}}(x))\varphi(x) \,dx = \int_\Omega a(x) h(w(x))\varphi(x) \,dx.
\end{equation}

Moreover, Lemma~\ref{apkdjoscvndfSHDKB<qrg4596uyh43io},
Corollary~\ref{PSjmndwqoeutr98n76bv9m0v04b8n9m-39tee45ndsh76y} and~\eqref{QUELLA} give that
\[
\big|(u_{s_{j_{k}}}(x)-w(x))  \mathcal L_{M,\varrho,{s_{j_{k}}}} \varphi(x)\big|\le\widetilde C_M\,\big(g(x)+|w(x)|\big)\,\|\varphi\|_{C^2(\R^n)},
\]for some~$\widetilde C_M>0$ independent of~$s$,
and the latter function belongs to~$L^1(\Omega)$.

Accordingly, by~\eqref{usaetou}, 
$$ \lim_{k\to+\infty}\int_{\R^n} (u_{s_{j_{k}}}(x)-w(x))\, \mathcal L_{M,\varrho,{s_{j_{k}}}} \varphi(x) \,dx=0.$$
Thus, by Proposition~\ref{Proplimito1} (which can be used here because~$w\in H^1(\R^n)$, thanks to~\eqref{QUESTA} and~\eqref{claimH1}),
\begin{eqnarray*}
&&\lim_{k\to+\infty}\frac{c_{n, {s_{j_{k}}}}}{2}\iint_{\R^n\times B_\varrho} \frac{(u_{s_{j_{k}}}(x)-u_{s_{j_{k}}}(x-y))(\varphi(x)-\varphi(x-y))}{|M(x-y, y) y|^{n+2{s_{j_{k}}}}} \,dx\, dy
\\&&\qquad=
\lim_{k\to+\infty}\int_{\R^n} u_{s_{j_{k}}}(x) \mathcal L_{M,\varrho,{s_{j_{k}}}} \varphi(x) \,dx=
\lim_{k\to+\infty}\int_{\R^n} w(x) \mathcal L_{M,\varrho,{s_{j_{k}}}} \varphi(x) \,dx \\
&&\qquad =\lim_{k\to+\infty}\frac{c_{n, {s_{j_{k}}}}}{2}\iint_{\R^n\times B_\varrho} \frac{(w(x)-w(x-y))(\varphi(x)-\varphi(x-y))}{|M(x-y, y) y|^{n+2{s_{j_{k}}}}} \,dx\, dy
\\&&\qquad= \int_{\R^n} A(x) \nabla w(x)\cdot \nabla\varphi(x) \,dx,
\end{eqnarray*}
being~$A(x)$ as in~\eqref{aijdefn}.  

Combining this and~\eqref{convergenzadx} with the fact that~$u_{s_{j_{k}}}$ is a weak solution of~\eqref{davidnonlocal2MS},
we obtain the desired claim in~\eqref{uastsolution}.

The estimates in~\eqref{GLinftyuBIS} follow from~\eqref{QUESTA} and~\eqref{claimH1}.\end{proof}

\section{Tailoring a matrix $M$ satisfying~\eqref{aijdefn} and proof of Theorems~\ref{THMFROMATOM}
and~\ref{fuweioifhqweio7564}}\label{SEC3}

\subsection{The algebraic roadmap leading to Theorem~\ref{THMFROMATOM}
and proof of Lemma~\ref{lemma3}}
We start by explicitly computing the following integral.

\begin{lemma}\label{lemma1}
Let~$\sigma_1, \dots, \sigma_n >0$. 
Then, for any~$i\in\{1,\dots, n\}$,
\[
\int_{\mathcal S^{n-1}}\frac{ \varphi_i^2}{\displaystyle \left(\sum_{k=1}^n \sigma^2_k \varphi^2_k\right)^{\frac{n+2}{2}}} \,
d\mathcal H_{\varphi}^{n-1} = \dfrac{V_{\mathbb E^n}}{\sigma_i^2},
\]
where
\begin{equation*}
V_{\mathbb E^n}:= \frac{\pi^{\frac{n}{2}}}{\Gamma\left(\frac{n+2}{2}\right) \sigma_1\dots\sigma_n}
\end{equation*}
denotes the volume of the~$n$-dimensional hyperellipsoid with semiaxes~$\frac{1}{\sigma_1},\dots,\frac{1}{\sigma_n}$.
\end{lemma}

\begin{proof} Up to reordering indices, we may suppose that~$i=1$.
We recall the system of spherical coordinates
(see e.g.~\cite[equation~(*) on page~65]{MR1530579}) given by
\begin{equation*}
\begin{split}
&\varphi_1= \cos x_1,\\
&\varphi_{j}= \cos x_j\,\prod_{k=1}^{j-1}\sin x_k,\qquad{\mbox{for all }}j\in\{2,\dots,n-2\},\\
&\varphi_{n-1}= \sin x_{n-1}\,\prod_{k=1}^{n-2}\sin x_k,\\
&\varphi_{n}= \cos x_{n-1}\,\prod_{k=1}^{n-2}\sin x_k,
\end{split}
\end{equation*}
with~$x_1,\dots,x_{n-2}\in[0,\pi]$ and~$x_{n-1}\in[0,2\pi]$.

The Jacobian determinant of this transformation
(see e.g.~\cite[page~66]{MR1530579}) can be written as 
\[
J = \prod_{k=1}^{n-2} \sin^{k}x_{n-1-k}.
\]
In addition, for any~$j\in\{1,\dots, n-1\}$, we set 
\begin{equation}\label{aidefn}
\alpha_j := \sqrt{\frac{\displaystyle\sum_{k=j+1}^{n} \sigma^2_k \varphi^2_k}{\displaystyle\prod_{k=1}^{j} \sin^2 x_k}}. 
\end{equation}
Using the short notation~$x^{(j)}:=(x_j,\dots,x_{n-1})$, we point out that
\begin{equation}\label{AODSCwojf-INAS}
{\mbox{$\alpha_j$ depends only on~$ x^{(j+1)}$.}}
\end{equation}

In this way, we have that~$\alpha_1$ does not depend on~$x_1$ and consequently
\begin{equation}\label{gyvubhijnkmol,}
\begin{split}
&{\mathcal{I}}:=\int_{\mathcal S^{n-1}}\frac{ \varphi_1^2 }{ \left(\displaystyle\sum_{k=1}^n \sigma^2_k \varphi^2_k\right)^{\frac{n+2}{2}}} \,d\mathcal H_{\varphi}^{n-1} = \int_{\mathcal S^{n-1}} \frac{\varphi^2_1 }{\left(\displaystyle
\sigma^2_1 \varphi^2_1+\sum_{{k=2}}^n \sigma^2_k\,\varphi_k^2 \right)^{\frac{n+2}{2}}}\, d\mathcal H^{n-1}_\varphi\\
&\qquad=\int_{[0,\pi]^{n-2}\times[0,2\pi]} \frac{\cos^2 x_1}{
\left(
\sigma^2_1 \cos^2 x_1
+ \alpha_1^2\sin^2 x_1\right)^{\frac{n+2}{2}}}\;
\prod_{k=1}^{n-2} \sin^{k}x_{n-1-k}\,dx^{(1)}.
\end{split}
\end{equation}

Now we claim that, for all~$m\in\{1,\dots,n-2\}$,
\begin{equation}\label{MAJ:QWDFV}
{\mathcal{I}}=\frac{\pi^{\frac{m}2}}{2\sigma_1^2\,\displaystyle\prod_{k=1}^m\sigma_k}\,
\frac{\Gamma\left(\frac{n-m}2\right)}{\Gamma\left(\frac{n+2}2\right)}
\int_{[0,\pi]^{n-m-2}\times[0,2\pi]}\frac1{\alpha_m^{n-m}}
\prod_{k=1}^{n-m-2} \sin^{k}x_{n-1-k}\,dx^{(m+1)}.
\end{equation}
Indeed, by~\cite[3.642, n.1, page~403]{MR3307944}, if~$a$, $b$, $\nu$, $\mu\in(0,+\infty)$ with~$2\nu-1\in2\N$,
\begin{equation}\label{4.5.gen}
\begin{split}
\int_{[0,\pi]} \frac{\cos^{2\nu-1} \tau\, \sin^{2\mu-1}\tau}{
\left(
b^2\cos^2 \tau
+ a^2\sin^2 \tau\right)^{\mu+\nu}}\,\,d\tau&
=2
\int_{0}^{\frac\pi2}  \frac{\cos^{2\nu-1} \tau\, \sin^{2\mu-1}\tau}{
\left(
b^2 \cos^2 \tau
+ a^2\sin^2 \tau\right)^{\mu+\nu}}\,d\tau\\
&= \frac{\Gamma\left(\nu\right)\,\Gamma\left(\mu\right)}{ a^{2\mu} b^{2\nu}\,\Gamma\left(\mu+\nu\right)}.
\end{split}
\end{equation}
By using this 
(with~$a:=\alpha_1$, $b:=\sigma_1$,
$\mu:=(n-1)/2$ and~$\nu:=3/2$, on the integration variable~$\tau:=x_1$), we find that
$$\int_{[0,\pi]} \frac{\cos^{2} x_1\, \sin^{n-2}x_1}{
\left(
\sigma_1^2\cos^2 x_1
+ \alpha_1^2\sin^2 x_1\right)^{\frac{n+2}2}}\,\,dx_1=
\frac{\sqrt{\pi}}{2 \alpha_1^{n-1} \sigma_1^{3}} \frac{\Gamma\left(\frac{n-1}2\right)}{\Gamma\left(\frac{n+2}2\right)}.$$
Plugging this information
into~\eqref{gyvubhijnkmol,}, we conclude that
\begin{equation*}
{\mathcal{I}}=
 \frac{\sqrt{\pi}}{2 \sigma^3_1} \frac{\Gamma\left(\frac{n-1}{2}\right)}{\Gamma\left(\frac{n+2}{2}\right)}
\int_{[0,\pi]^{n-3}\times[0,2\pi]}
\frac{1}{\alpha_1^{n-1}}
\;
\prod_{k=1}^{n-3} \sin^{k}x_{n-1-k}\,dx^{(2)}.\end{equation*}
This proves~\eqref{MAJ:QWDFV} with~$m=1$.

Now we argue recursively, supposing that~\eqref{MAJ:QWDFV} holds true for the index~$m$,
with~$m\le n-3$, and we establish it for the index~$m+1$. To accomplish this goal, we point out that, for all~$j\in\{1,\dots,n-3\}$,
\begin{equation*}
\begin{split}&
\alpha_j^2= \frac{\displaystyle\sum_{k=j+1}^{n} \sigma^2_k \varphi^2_k}{\displaystyle\prod_{k=1}^{j} \sin^2 x_k}
=\frac{ \sigma^2_{j+1} \varphi^2_{j+1}+
\displaystyle\sum_{k=j+2}^{n} \sigma^2_k \varphi^2_k}{\displaystyle\prod_{k=1}^{j} \sin^2 x_k}=
\frac{ \sigma^2_{j+1} \varphi^2_{j+1}+\alpha_{j+1}^2
\displaystyle\displaystyle\prod_{k=1}^{j+1} \sin^2 x_k }{\displaystyle\prod_{k=1}^{j} \sin^2 x_k}\\&\qquad\qquad=
\sigma_{j+1}^2\cos^2 x_{j+1}+\alpha_{j+1}^2\sin^2 x_{j+1}.
\end{split}\end{equation*}
This and the inductive hypothesis yield
\begin{equation}\label{DFDvb-09o4wkhgb90wriuhjnp323wed6ns8oh}
\begin{split}
{\mathcal{I}}&=\frac{\pi^{\frac{m}2}}{2\sigma_1^2\,\displaystyle\prod_{k=1}^m\sigma_k}\,
\frac{\Gamma\left(\frac{n-m}2\right)}{\Gamma\left(\frac{n+2}2\right)}\\&\quad\times
\int_{[0,\pi]^{n-m-2}\times[0,2\pi]}\frac1{\left( 
\sigma_{m+1}^2\cos^2 x_{m+1}+\alpha_{m+1}^2\sin^2 x_{m+1}
\right)^{\frac{n-m}2}}
\prod_{k=1}^{n-m-2} \sin^{k}x_{n-1-k}\,dx^{(m+1)}.
\end{split}\end{equation}

Besides, by~\eqref{AODSCwojf-INAS}, $\alpha_{m+1}$ depends only on~$ x^{(m+2)}$
and we can thereby utilize~\eqref{4.5.gen} (here with~$a:=\alpha_{m+1}$, $b:=\sigma_{m+1}$,
$\mu:=\frac{n-m-1}{2}$ and~$\nu:=1/2$) and we see that
\begin{equation*}
\begin{split}
\int_{[0,\pi]} \frac{ \sin^{n-m-2}x_{m+1}}{
\left(
\sigma_{m+1}^2\cos^2 x_{m+1}
+ \alpha_{m+1}^2\sin^2 x_{m+1}\right)^{\frac{n-m}2}}\,\,dx_{m+1}
= \frac{\sqrt{\pi}}{ \alpha_{m+1}^{n-m-1} \sigma_{m+1}} \frac{\Gamma\left(\frac{n-m-1}2\right)}{\Gamma\left(\frac{n-m}2\right)}.
\end{split}
\end{equation*}

{F}rom this and~\eqref{DFDvb-09o4wkhgb90wriuhjnp323wed6ns8oh} we arrive at
\begin{equation*}
\begin{split}
{\mathcal{I}}&=\frac{\pi^{\frac{m}2}}{2\sigma_1^2\,\displaystyle\prod_{k=1}^m\sigma_k}\,
\frac{\Gamma\left(\frac{n-m}2\right)}{\Gamma\left(\frac{n+2}2\right)}
\frac{\sqrt{\pi}}{ \sigma_{m+1}} \frac{\Gamma\left(\frac{n-m-1}2\right)}{\Gamma\left(\frac{n-m}2\right)}
\\&\qquad\times
\int_{[0,\pi]^{n-m-3}\times[0,2\pi]}\frac{1}{\alpha_{m+1}^{n-m-1}}
\prod_{k=1}^{n-m-3} \sin^{k}x_{n-1-k}\,dx^{(m+2)}
,\end{split}\end{equation*}
which gives~\eqref{MAJ:QWDFV} for the index~$m+1$. This completes the inductive step and establishes~\eqref{MAJ:QWDFV}.

We also remark that
\begin{eqnarray*}
\alpha_{n-2}^2=\frac{\displaystyle\sum_{k=n-1}^{n} \sigma^2_k \varphi^2_k}{\displaystyle\prod_{k=1}^{n-2} \sin^2 x_k}=
\frac{\sigma^2_{n-1} \varphi^2_{n-1}+\sigma^2_{n} \varphi^2_{n}}{\displaystyle\prod_{k=1}^{n-2} \sin^2 x_k}=\sigma^2_{n-1}\sin^2x_{n-1}+\sigma^2_{n}\cos^2 x_{n-1}.
\end{eqnarray*}
Thus, we deduce from~\eqref{MAJ:QWDFV} (used with~$m:=n-2$) that
\begin{equation}\label{2P3eokjfbtr-pypGHk7-4tjyp}
\begin{split}
{\mathcal{I}}&=\frac{\pi^{\frac{n-2}2}}{2\,\sigma_1^2\,\displaystyle\prod_{k=1}^{n-2}\sigma_k}\,
\frac{1}{\Gamma\left(\frac{n+2}2\right)}
\int_{[0,2\pi]}\frac1{\alpha_{n-2}^2}\,dx_{n-1}\\&=\frac{\pi^{\frac{n-2}2}}{2\sigma_1^2\,\displaystyle\prod_{k=1}^{n-2}\sigma_k}\,
\frac{1}{\Gamma\left(\frac{n+2}2\right)}
\int_{[0,2\pi]}\frac1{\sigma^2_{n-1}\sin^2x_{n-1}+\sigma^2_{n}\cos^2 x_{n-1}}\,dx_{n-1}.
\end{split}
\end{equation}

Also, from~\eqref{4.5.gen} (used here with~$a:=\sigma_{n-1}$, $b:=\sigma_n$ and~$\mu:=\nu:=1/2$), 
\begin{eqnarray*}\int_{[0,2\pi]}\frac1{\sigma^2_{n-1}\sin^2x_{n-1}+\sigma^2_{n}\cos^2 x_{n-1}}\,dx_{n-1}&=& 2
\int_{[0,\pi]}\frac1{\sigma^2_{n-1}\sin^2x_{n-1}+\sigma^2_{n}\cos^2 x_{n-1}}\,dx_{n-1}\\&
=&\frac{2{\pi}}{ \sigma_{n-1}\sigma_{n}}.\end{eqnarray*}
Combining this and~\eqref{2P3eokjfbtr-pypGHk7-4tjyp}, the desired result follows.
\end{proof}

\begin{lemma}\label{lemma2}
Let~$\lambda_1, \dots, \lambda_n >0$. Let~$\overline\lambda:=\lambda_1\cdots\lambda_n$ and, for any~$i\in\{1,\dots,n\}$,
\begin{equation}\label{fuiwyr9034fujdsgvauewyrt98043ooiuyt}
\sigma_i:=\sqrt{\frac{\overline\lambda^{\frac{1}{n+2}}}{\lambda_i}}. \end{equation}
Let also
\[
V_n:=\dfrac{\pi^{\frac{n}{2}}}{\Gamma\left(\frac{n+2}{2}\right)}
\]
denote the volume of the~$n$-dimensional unit ball.

Then, for any~$i$, $j\in\{1,\dots, n\}$, 
\begin{equation}\label{claimlambdai}
\frac{1}{V_n}\int_{\mathcal S^{n-1}} \frac{\varphi_i\varphi_j }{ \displaystyle \left(\sum_{k=1}^n \sigma^2_k \varphi^2_k\right)^{\frac{n+2}{2}} } \,d\mathcal H_{\varphi}^{n-1} = \lambda_i\delta_{ij}.
\end{equation}

In addition,
\begin{equation}\label{inequalitysigmai}
\dfrac{\displaystyle\left(\min_{j\in\{1,\dots,n\}} \lambda_j\right)^{\frac{n}{2(n+2)}}}{\displaystyle\left(\max_{j\in\{1,\dots,n\}} \lambda_j\right)^{\frac12}}\le \sigma_i\le \dfrac{\displaystyle\left(\max_{j\in\{1,\dots,n\}} \lambda_j\right)^{\frac{n}{2(n+2)}}}{\displaystyle\left(\min_{j\in\{1,\dots,n\}} \lambda_j\right)^{\frac12}}.
\end{equation}
\end{lemma}

\begin{proof}
We set~$\overline\sigma:=\sigma_1\dots\sigma_n$ and we deduce from~\eqref{fuiwyr9034fujdsgvauewyrt98043ooiuyt} that
\begin{eqnarray*}
\overline\sigma = \frac{\overline\lambda^{\frac{n}{2(n+2)}}}{\overline\lambda^{\frac12}}= \frac{1}{\overline\lambda^{\frac{1}{n+2}}}.
\end{eqnarray*}
As a consequence
$$ \lambda_i=\frac{\overline\lambda^{\frac{1}{n+2}}}{\sigma_i^2}=\frac{1}{\sigma^2_i \overline\sigma}.
$$
{F}rom this and Lemma~\ref{lemma1}, we obtain that
\begin{eqnarray*}
\frac{1}{V_n}\int_{\mathcal S^{n-1}} \frac{\varphi_i^2 }{ \displaystyle \left(\sum_{k=1}^n \sigma^2_k \varphi^2_k\right)^{\frac{n+2}{2}} } \,d\mathcal H_{\varphi}^{n-1}=\frac{V_{\mathbb E^n}}{\sigma_i^2V_n}=\frac{1}{\sigma^2_i \overline\sigma}=\lambda_i,
\end{eqnarray*}
which is the desired identity in~\eqref{claimlambdai} when~$i=j$.

If instead~$i\neq j$, both sides of~\eqref{claimlambdai} vanish,
due to the odd symmetry of the integrand in the~$i$-th coordinate.

Moreover, that inequalities in~\eqref{inequalitysigmai} plainly follow from~\eqref{fuiwyr9034fujdsgvauewyrt98043ooiuyt}.
\end{proof}

We can now prove Lemma~\ref{lemma3}:

\begin{proof}[Proof of Lemma~\ref{lemma3}]
By~\eqref{AOLAMBDAOT} and~\eqref{AOLAMBDAOT2}, we have that 
\[
A_{ij} = \sum_{k, l=1}^n O_{ik}\,\Lambda_{kl} \, O_{lj}^T = \sum_{k,l=1}^n O_{ik} \lambda_k \delta_{kl} O_{jl}.
\]
{F}rom this and Lemma~\ref{lemma2}, we deduce that
\begin{equation}\label{ygvbhunijokml,ò}
A_{ij} = \dfrac{n}{\omega_{n-1}}
\int_{\mathcal S^{n-1}} \frac{\displaystyle \sum_{k,l=1}^n
O_{ik}  \varphi_k\varphi_l O_{jl} }{\displaystyle \left|\sum_{l=1}^n \sigma^2_l \varphi^2_l\right|^{\frac{n+2}{2}} } \,d\mathcal H_{\varphi}^{n-1},
\end{equation}
where we have also used the identity~$\omega_{n-1}=n V_n$.

Now we perform the change of variable~$\psi:= O\varphi$. In light of~\eqref{MOSIGMAOT}, we see that
\begin{equation}\label{ygvbhunijokml,ò2}\begin{split}&
|N_A\psi|^2= (N_A\psi)^T (N_A\psi)= \psi^T N_A^T N_A \psi=
\psi^T (O\Sigma O^T)^T (O\Sigma O^T) \psi\\&\qquad
= \psi^T O\Sigma^2 O^T \psi
 = \varphi^T \Sigma^2 \varphi =
\sum_{l=1}^n \sigma^2_l \varphi^2_l .
\end{split}\end{equation}

Also,
\begin{equation}\label{ygvbhunijokml,ò2gerher}
\sum_{k,l=1}^nO_{ik}\varphi_k\varphi_l O_{jl} = (O\varphi)_i (O\varphi)_j=\psi_i\psi_j.
\end{equation}
Moreover,
we observe that, since~$\varphi\in\mathcal S^{n-1}$ and~$O$ is orthogonal,  we have that~$d\mathcal H^{n-1}_\psi = d\mathcal H^{n-1}_\varphi$.

Accordingly, from this, \eqref{ygvbhunijokml,ò}, \eqref{ygvbhunijokml,ò2}
and~\eqref{ygvbhunijokml,ò2gerher}, we obtain that
\[
A_{ij} =\dfrac{n}{\omega_{n-1}}\int_{\mathcal S^{n-1}} \frac{\psi_i \,\psi_j}{|N_A \psi|^{n+2}}\, d\mathcal H^{n-1}_\psi,
\]
as desired.

\end{proof}

\subsection{Proof of Theorem~\ref{THMFROMATOM}}
We observe that a matrix-valued satisfying~\eqref{matdefpos} and~\eqref{structuralM} can be reconstructed by the mere knowledge of~$M(x, 0)$, as we now discuss.

\begin{lemma}\label{lemma5}
Let~$N:\R^n\to\R^{n\times n}$ and let~$H:\R^n\to\R^{n\times n}$ be such that
\begin{equation}\label{ipotesiP}
H(0)=0\quad\mbox{ and }\quad H(y)=H(-y) \ \mbox{ for any $y\in\R^n$.}
\end{equation}
For any~$x$, $ y\in\R^n$, define~$M:\R^n\times\R^n\to\R^{n\times n}$ as
\begin{equation*}
M(x, y):= \dfrac{N(x)+N(x+y)}{2} + H(y).
\end{equation*}

Then,
\[
M(x, 0)=N(x) \quad\mbox{ and }\quad M(x-y, y) = M(x, -y).
\]
\end{lemma}

\begin{proof}
In light of~\eqref{ipotesiP}, we have that~$M(x, 0)=N(x)$ and that 
\begin{eqnarray*}
M(x-y, y)- M(x, -y)=\dfrac{N(x-y)+N(x)}{2} + H(y) - \dfrac{N(x)+N(x-y)}{2} + H(-y) =0,
\end{eqnarray*} as desired.
\end{proof}

\begin{lemma}\label{lemmaNAx}
Let~$\Sigma:\R^n\to\R^{n\times n}$ be a matrix-valued function with~$\Sigma(x)$ diagonal and such that
\begin{equation}\label{guyijnkolmò,à.2}0<\inf_{{x\in\Omega}\atop{j\in\{1,\dots,n\}}}\Sigma_{jj}(x)\le
\sup_{{x\in\Omega}\atop{j\in\{1,\dots,n\}} }\Sigma_{jj}(x)<+\infty.\end{equation}

Let~$O:\R^n\to\R^{n\times n}$ be such that~$O(x)$ is an orthogonal matrix and define
\begin{equation}\label{guyijnkolmò,à.}
N(x): = O(x) \Sigma(x) O^T(x).
\end{equation}

Then, there exist $\widetilde\alpha\ge\widetilde\beta>0$ such that, for any $x$, $\xi\in\R^n$,
\[
\widetilde\beta |\xi|^2\le N(x) \xi\cdot\xi\le\widetilde\alpha |\xi|^2.
\]
\end{lemma}
\begin{proof} We use the notation~$\Sigma_{ij}(x):= \sigma_i(x)\delta_{ij}$.
Moreover, for any $x, \xi\in\R^n$, we set
\[
\eta(x):= O^T(x)\xi.
\]
Since $O(x)$ is orthogonal, we have $|\eta(x)|=|\xi|$ for any $x\in\R^n$. Hence, combining \eqref{guyijnkolmò,à.2} and~\eqref{guyijnkolmò,à.}, we see that, for any $x\in\R^n$,
\begin{equation}\label{ubijnkml,òcDCNok2}
\begin{split}
N(x) \xi\cdot\xi &= \xi^T O(x) \Sigma(x) O^T(x) \xi = \eta^T(x) \Sigma(x) \eta(x) = \sum_{i=1}^n \sigma_i(x) \eta^2_i(x)\\
&\in\big[\widetilde\beta |\eta(x)|^2,\,\widetilde\alpha |\eta(x)|^2\big] =\big[\widetilde\beta|\xi|^2,\, \widetilde\alpha |\xi|^2\big],
\end{split}
\end{equation}
for some $\widetilde\alpha\ge\widetilde\beta>0$ independent of $x$.
\end{proof}

\begin{lemma}\label{lemmajdncbJò}
Let~$N$ be as in Lemma~\ref{lemmaNAx}. Let~$M$ and~$H$ be as in Lemma~\ref{lemma5}.

Assume also that for all~$x$, $\xi\in\R^n$,
\begin{equation}\label{q-0wfohgw0i9oeitgh40-p5yjotuk-01}0\le H(x)\xi\cdot\xi\le C_0,\end{equation}
for some~$C_0\in(0,+\infty)$.

Then,  there exist $\alpha\ge \beta>0$, depending only on~$\widetilde\alpha$, $\widetilde\beta$, $C_0$ and~$n$,
such that, for all~$x$, $y$, $\xi\in\R^n$,
\begin{equation}\label{q-0wfohgw0i9oeitgh40-p5yjotuk-02}
\beta|\xi|^2\le M(x-y, y)\xi\cdot\xi\le \alpha|\xi|^2
\end{equation}
and
\begin{equation}\label{q-0wfohgw0i9oeitgh40-p5yjotuk-03}
\beta|\xi| \le |M(x-y, y)\xi|\le \alpha|\xi|.
\end{equation}
\end{lemma}
\begin{proof}We have that
\begin{equation}\label{uihfesdvjkò1}
M(x-y, y)\xi\cdot\xi =\frac12 (N(x-y)+N(x))\xi\cdot\xi + H(y)\xi\cdot\xi .
\end{equation}
Hence, the desired result in~\eqref{q-0wfohgw0i9oeitgh40-p5yjotuk-02} follows from~\eqref{q-0wfohgw0i9oeitgh40-p5yjotuk-01} and Lemma \ref{lemmaNAx}.

We can thereby apply Lemma \ref{propnormeequiv00} and
deduce~\eqref{q-0wfohgw0i9oeitgh40-p5yjotuk-03}
from~\eqref{q-0wfohgw0i9oeitgh40-p5yjotuk-02} (up to renaming~$\alpha$ and~$\beta$).\end{proof}

We now complete the proof of Theorem~\ref{THMFROMATOM}.

\begin{proof}[Proof of Theorem~\ref{THMFROMATOM}]
Since~$A(x)$ is symmetric, by~\cite[Theorem~5S, page~330]{STRANG4} we have that
\begin{equation}\label{AOLAMBDAOTXMAIN}
A(x):=O(x)\Lambda(x) O^T(x),
\end{equation}
where~$O(x)$, $\Lambda(x):\R^n\to\R^{n\times n}$ satisfy
\[
O(x)O^T(x) = \operatorname{Id}\quad\mbox{ and }\quad \Lambda_{ij}(x):= \lambda_i(x)\delta_{ij}\quad\mbox{ for any } x\in\R^n.
\]
Since~$A$ is continuous, so are~$\Lambda$ and~$O$, due to Lemmata~\ref{WERFV} and~\ref{lemmafeye542yt7jksdjklinear}.
Hence, if~$\Sigma(x)$ and~$N_{A(x)}$ are defined as in Lemma~\ref{lemma3},
we have that~$\Sigma$ is a bounded and continuous matrix-valued function (recall also~\eqref{tildelambdaMAIN}),
and thus so is~$N_{A(x)}$.  

Now, by using Lemmata~\ref{lemma5} and~\ref{lemmajdncbJò} we infer that there exists a matrix~$M(x,y)$, with~$x\in\R^n$ and~$y\in B_\varrho$,
satisfying~\eqref{matdefpos} and~\eqref{structuralM} for any $\varrho\in(0,+\infty]$.

Also, the claim in~\eqref{aijdefn}
follows from~\eqref{20ef-wdfuvjb9ogrbtg}.

We remark that the assumption that~$N_{A(x)}$ and $H$ are continuous allow us to define $M(x, 0)$.
\end{proof}

\subsection{Reconstructing the classical solution through the limit~$s\nearrow 1$ and proof of Theorem~\ref{fuweioifhqweio7564}}

We now complete the proof of Theorem~\ref{fuweioifhqweio7564}.

\begin{proof}[Proof of Theorem~\ref{fuweioifhqweio7564}]
Given~$A:\Omega\to\R^{n\times n}$, we extend it in a small neighborhood containing the closure of~$\Omega$ (e.g.,
by orthogonal projection) and consider a mollification of all its entries, which will be denoted by~$A_\ell$, with~$\ell\in\N$. In this way, we have that\begin{equation*} \lim_{\ell\to+\infty}\|A-A_\ell\|_{L^1(\Omega)}=0\end{equation*}
and, up to a (not relabeled) subsequence, for a.e.~$x\in\Omega$,\begin{equation}\label{ASKJDEIKF:X} \lim_{\ell\to+\infty}A_\ell(x)=A(x).\end{equation}
We observe that, denoting by~$\eta_\ell\in C^\infty_0(\R^n)$ the mollifier used in this procedure, for all~$\xi\in\mathcal S^{n-1}$,\begin{equation}\label{AKPDALLDFK} A_\ell(x)\xi\cdot\xi=\int_{\R^n} A(y) \xi\cdot\xi \eta_\ell(x-y)\,dy\in\left[\undertilde\lambda\int_{\R^n} \eta_\ell(x-y)\,dy,\,\widetilde\lambda\int_{\R^n} \eta_\ell(x-y)\,dy\right]=[\undertilde\lambda,\widetilde\lambda].\end{equation}

Now, we diagonalize the symmetric matrix~$A(x)$ by writing, for all~$x\in\Omega$, that\begin{equation}\label{PSJLMDIHFKNOJLMfrj}A(x)=O(x)\Lambda(x) O^T(x),\end{equation} where~$\Lambda(x)$ is diagonal and~$O(x)$ is orthogonal.

We now utilize Lemma~\ref{lemmafeye542yt7jksdjklinear} with~$B:=A_\ell(x)$. Specifically, by~\eqref{ASKJDEIKF:X}, for a.e.~$x\in\Omega$ we find~$\ell_x\in\N$ such that, if~$\ell\ge \ell_x$, then~$\|A(x)-A_\ell(x)\|\le c$, being~$c>0$ the structural constant in Lemma~\ref{lemmafeye542yt7jksdjklinear}. Accordingly, combining~\eqref{PSJLMDIHFKNOJLMfrj} and Lemma~\ref{lemmafeye542yt7jksdjklinear},
for all~$\ell\ge\ell_x$ we find an orthogonal matrix~$O_\ell(x)$ such as~$\Lambda_\ell(x):=O_\ell^T(x)A_\ell(x)O_\ell(x)$ is diagonal and$$\|O(x)-O_\ell(x)\|\le C\|A(x)-A_\ell(x)\|.$$
Actually, by virtue of Lemma~\ref{WERFV}, we also know that
\begin{equation*}\lim_{\ell\to+\infty}\|\Lambda_\ell(x)-\Lambda(x)\|=0.\end{equation*}

Owing to~\eqref{AKPDALLDFK}, if~$\lambda_{1,\ell}\le\dots\le\lambda_{n,\ell}$ are the eigenvalues of~$A_\ell$, we know that~$\lambda_{1,\ell}, \dots, \lambda_{n,\ell}\in[\undertilde\lambda,\widetilde\lambda]$. 
Then, we are in the position of exploiting Theorem~\ref{THMFROMATOM} (with~$A:=A_\ell$) and obtain that
there exists $M_\ell:\R^n\times\R^n\to\R^{n\times n}$ satisfying~\eqref{matdefpos}, \eqref{structuralM} and~\eqref{aijdefn} (for any $\varrho\in (0, +\infty]$).

In particular, recalling~\eqref{0wqurj3o-4uy4ijp30to5Xmm0vb5vn} (used here with~$H:=0$), we have that
\begin{equation}\label{OJSLN0djewglht:0rjtmth}M_\ell(x, y)= \dfrac{N_{A_\ell(x)}+N_{A_\ell(x+y)}}{2}.\end{equation}

We now recall the notation in~\eqref{LUoperator}
and denote by~$u_{s,\ell}\in H_0^s(\Omega)\cap L^\infty(\Omega)$ the weak solution of
\[
\begin{cases}
\mathcal L_{M_\ell,\varrho, s} u + a h(u) = f &\mbox{ in } \Omega\\
u=0 &\mbox{ in } \R^n\setminus\Omega
\end{cases}
\]
whose existence and uniqueness is guaranteed by Theorem~\ref{main2} (for any $\varrho\in (0, +\infty]$).
Bearing in mind~\eqref{OJSLN0djewglht:0rjtmth} and the smoothness of~$A_\ell$, we see that the assumption in~\eqref{ASSUNZ:MM} is satisfied,
allowing us to use Theorem~\ref{main3} and conclude that\[ \lim_{s\nearrow 1} u_{s,\ell}(x) = u_{1,\ell}(x),\]where~$u_{1,\ell}\in H^1_0(\Omega)\cap L^\infty(\Omega)$ solves
\[
\begin{cases}
-\operatorname{div}(A_\ell\nabla u) + a h(u) = f &\mbox{ in } \Omega\\
u=0 &\mbox{ on } \partial\Omega.
\end{cases}
\]
Moreover, by~\eqref{GLinftyuBIS},
\begin{equation}\label{unifyegwca4353yddbwek}\begin{split}&
\|u_{1,\ell}\|_{L^\infty(\Omega)}\le h^{-1}(Q) \\
{\mbox{and }}\quad &\|\nabla u_{1,\ell}\|_{L^2(\Omega)}^2\le
C\,h^{-1}(Q) \,\big(\|f\|_{L^1(\Omega)} + Q \|a\|_{L^1(\Omega)}\big).
\end{split}\end{equation}

We can therefore assume, up to a subsequence, that~$u_{1,\ell}$ converges weakly in~$H^1_0(\Omega)$ to some function~$u_1$. Passing to the limit the weak equation satisfied by~$u_{1,\ell}$ and the uniform estimate in~\eqref{unifyegwca4353yddbwek}, the desired result follows.
\end{proof}

\begin{appendix}

\section{Calculating~${\mathcal{L}}_K$ in the pointwise sense}
Here we show that, for $s\in (0, \frac12)$, the principal value in definition \eqref{qsdc:OPK} is not needed and the operator $\mathcal L_K u$ is well-defined for smooth and compactly supported functions.

\begin{lemma}
Let $s\in (0, \frac12)$ and assume that \eqref{QrehtgDWFfbg0jolwfdv:Qwfgb} holds.

Then, for all~$u\in C^\infty_c(\R^n)$ and~$x\in\R^n$, the map~$\R^n\ni y\longmapsto\big(u(x) - u(x-y)\big)\,K(x,x-y)$
is in~$L^1(\R^n)$ and
\[
\|{\mathcal{L}}_Ku\|_{L^\infty(\R^n)}\le C\, C_n \,\|u\|_{C^1(\R^n)}\left(\frac1s +\frac{1}{1-2s}\right) 
\]
where $C_n>0$ depends only on~$n$, and~$C$ is as in~\eqref{QrehtgDWFfbg0jolwfdv:Qwfgb}.
\end{lemma}

\begin{proof} We have that
\[
\int_{\R^n\setminus B_1} |u(x) - u(x-y)|\,K(x,x-y)\, dy\le 2C \|u\|_{L^\infty(\R^n)} \int_{\R^n\setminus B_1} \frac{dy}{|y|^{n+2s}} = \frac{C\,\widetilde C}{s} \|u\|_{L^\infty(\R^n)},
\]
for some $\widetilde C>0$ depending only on $n$.

Moreover,  since $s\in (0, \frac12)$, we infer that
\[
\begin{split}
\int_{B_1} |u(x) - u(x-y)|\,K(x,x-y)\, dy\le C\,\widehat C\,\|\nabla u\|_{L^\infty(\R^n)}\int_{B_1} \frac{dy}{|y|^{n+2s-1}} =\frac{C\,\widehat C}{1-2s} \|\nabla u\|_{L^\infty(\R^n)},
\end{split}
\]
for some $\widehat C>0$ depending only on $n$ and possibly changing during the computation.
\end{proof}

Here we point out that~${\mathcal{L}}_K$ can be computed pointwise on smooth functions,
provided that condition~\eqref{COND:qsdfdv} is assumed.

\begin{lemma}\label{apkdjoscvndfSHDKB<qrg4596uyh43io}
If~$s\in\left[ \frac12, 1\right)$ and condition~\eqref{COND:qsdfdv} holds true, then the expression in~\eqref{qsdc:OPK}
is well-defined for all~$u\in C^\infty_c(\R^n)$ and
$$ \|{\mathcal{L}}_Ku\|_{L^\infty(\R^n)}\le C_n\,\|u\|_{C^2(\R^n)}\left(
\frac{C}{s(1-s)}+C_\sharp
\right),$$
where~$C_n>0$ depends only on~$n$, $C$ is as in~\eqref{QrehtgDWFfbg0jolwfdv:Qwfgb}
and~$C_\sharp$ denotes the left-hand side of~\eqref{COND:qsdfdv}.
\end{lemma}

\begin{proof} Let~$\varepsilon\in(0,1)$ and observe that
\begin{equation}\label{ofhg0trypmik-3t4h[tprJHm-2rithkj-2034-2}\begin{split}
& 2\int_{B_1\setminus B_\varepsilon} \big(u(x) - u(x-y)\big)\,K(x,x-y)\, dy
\\&\quad=\int_{B_1\setminus B_\varepsilon} \big(u(x) - u(x+y)\big)\,K(x,x+y)\, dy
+\int_{B_1\setminus B_\varepsilon} \big(u(x) - u(x-y)\big)\,K(x,x-y)\, dy
\\&\quad=\int_{B_1\setminus B_\varepsilon} \big(u(x) - u(x+y)\big)\,\Big(K(x,x+y)-K(x,x-y)\Big)\, dy
\\&\qquad\qquad+\int_{B_1\setminus B_\varepsilon} \big(2u(x) -u(x+y)- u(x-y)\big)\,K(x,x-y)\, dy.
\end{split}\end{equation}

Moreover, in view of~\eqref{QrehtgDWFfbg0jolwfdv:Qwfgb}, for all~$x\in\R^n$,
\begin{equation*}\begin{split}
\big|2u(x) -u(x+y)- u(x-y)\big|\,K(x,x-y)&\le
\frac{\| D^2u\|_{L^\infty(\R^n)}}2 \,|y|^2\,K(x,x-y)\\&\le
\frac{C\,\| D^2u\|_{L^\infty(\R^n)}}{2\,|y|^{n+2s-2}}.\end{split}
\end{equation*}
The latter expression belongs to~$L^1(B_1)$ and thus
\begin{equation}\label{ofhg0trypmik-3t4h[tprJHm-2rithkj-2034-3}\begin{split}&
\lim_{\varepsilon\searrow0}
\int_{B_1\setminus B_\varepsilon} \big(2u(x) -u(x+y)- u(x-y)\big)\,K(x,x-y)\, dy
\\&\quad=
\int_{B_1} \big(2u(x) -u(x+y)- u(x-y)\big)\,K(x,x-y)\, dy\\&\quad\in
\left[-\frac{C_n\,C\,\| D^2u\|_{L^\infty(\R^n)}}{1-s}, \,\frac{C_n\,C\,\| D^2u\|_{L^\infty(\R^n)}}{1-s}\right],
\end{split}\end{equation}
for some~$C_n>0$.

Furthermore, 
$$ \big(u(x) - u(x+y)\big)\,\Big(K(x,x+y)-K(x,x-y)\Big)\le
\|\nabla u\|_{L^\infty(\R^n)}\,|y|\,\Big(K(x,x+y)-K(x,x-y)\Big).$$
In view of~\eqref{COND:qsdfdv}, this expression is integrable over~$y\in B_1$,
with
\begin{equation}\label{ofhg0trypmik-3t4h[tprJHm-2rithkj-2034-4}\begin{split}&
\lim_{\varepsilon\searrow0}
\int_{B_1\setminus B_\varepsilon} \big(u(x) - u(x+y)\big)\,\Big(K(x,x+y)-K(x,x-y)\Big)\, dy\\&\quad=\int_{B_1} \big(u(x) - u(x+y)\big)\,\Big(K(x,x+y)-K(x,x-y)\Big)\, dy\\&\quad
\in\left[-C_\sharp\,\| \nabla u\|_{L^\infty(\R^n)},\,
C_\sharp\,\| \nabla u\|_{L^\infty(\R^n)}\right].
\end{split}\end{equation}

Also,
$$ \big|u(x) - u(x-y)\big|\,K(x,x-y)\le \frac{2C\,\|u\|_{L^\infty(\R^n)}}{|y|^{n+2s}},$$
which is integrable over~$y\in \R^n\setminus B_1$, with
\begin{equation}\label{ofhg0trypmik-3t4h[tprJHm-2rithkj-2034-5}
\left|\;\int_{\R^n\setminus B_1} \big(u(x) - u(x-y)\big)\,K(x,x-y)\, dy\right|\le
\frac{C_n\,C\,\|u\|_{L^\infty(\R^n)}}{s}
,\end{equation}
up to renaming~$C_n$.

The desired result now follows from~\eqref{ofhg0trypmik-3t4h[tprJHm-2rithkj-2034-2},
\eqref{ofhg0trypmik-3t4h[tprJHm-2rithkj-2034-3}, \eqref{ofhg0trypmik-3t4h[tprJHm-2rithkj-2034-4}
and~\eqref{ofhg0trypmik-3t4h[tprJHm-2rithkj-2034-5}.
\end{proof}

\section{Bounds on nondegenerate matrices}

We now discuss some bounds involving nondegenerate matrices.

\begin{lemma}\label{lemmaduewqi1234567poiuyt}
Let~$\alpha\ge\beta>0$.
Let~${\mathcal{M}}_{\alpha,\beta}$ be the collection of
matrices~$L\in\R^{n\times n}$ such that, for all~$\xi\in \mathcal{S}^{n-1}$,
$$ L\xi\cdot\xi\in [\beta,\alpha].$$

Then,
\begin{equation}\label{ALqwdf-2we-1}
\inf_{{L\in{\mathcal{M}}_{\alpha,\beta}}\atop{\xi\in\R^n\setminus\{0\}}}\frac{|L\xi|}{|\xi|}
=\min_{{L\in{\mathcal{M}}_{\alpha,\beta}}\atop{\xi\in\R^n\setminus\{0\}}}\frac{|L\xi|}{|\xi|}
>0
\end{equation}
and
\begin{equation}\label{ALqwdf-2we-2}
\sup_{{L\in{\mathcal{M}}_{\alpha,\beta}}\atop{\xi\in\R^n\setminus\{0\}}}\frac{|L\xi|}{|\xi|}
=\max_{{L\in{\mathcal{M}}_{\alpha,\beta}}\atop{\xi\in\R^n\setminus\{0\}}}\frac{|L\xi|}{|\xi|}
<+\infty.
\end{equation}
\end{lemma}

\begin{proof} We focus on the proof of~\eqref{ALqwdf-2we-1}, since the proof of~\eqref{ALqwdf-2we-2} is alike.

We remark that~${\mathcal{M}}_{\alpha,\beta}\times \mathcal{S}^{n-1}$ is a closed and bounded subset of the Euclidean space~$\R^{n\times n+n}$ and therefore the following infimum is attained
$$ \inf_{{L\in{\mathcal{M}}_{\alpha,\beta}}\atop{\xi\in\mathcal{S}^{n-1}}} |L\xi|.$$
Hence, we can find~$L_\star\in \R^{n\times n}$ and~$\xi_\star\in \mathcal{S}^{n-1}$ such that
\begin{equation}\label{nqi0zodgifvn-r1}|L_\star \xi_\star|= \inf_{{L\in{\mathcal{M}}_{\alpha,\beta}}\atop{\xi\in\mathcal{S}^{n-1}}} |L\xi|=\inf_{{L\in{\mathcal{M}}_{\alpha,\beta}}\atop{\xi\in\R^n\setminus\{0\}}}\frac{|L\xi|}{|\xi|}.\end{equation}
We also notice that \begin{equation}\label{nqi0zodgifvn-r2}
L_\star\xi_\star\ne0,\end{equation} otherwise~$0=L_\star\xi_\star\cdot\xi_\star\ge\beta$, which is a contradiction.

The desired result in~\eqref{ALqwdf-2we-1} now follows from~\eqref{nqi0zodgifvn-r1} and~\eqref{nqi0zodgifvn-r2}.
\end{proof}

\begin{lemma}\label{propnormeequiv00}
Let~$\varrho\in(0,+\infty]$ and~$M:\R^n\times B_\varrho\to\R^{n\times n}$ and assume that there exist $\alpha\ge\beta>0$ such that, for any $x\in\R^n$, $y\in B_\varrho$, $\xi\in\R^N$,
\begin{equation}\label{tfyugbhijnkml,ò}
\beta |\xi|^2\le M(x-y, y)\xi\cdot\xi\le \alpha |\xi|^2.
\end{equation}

Then, there exist~$0<c_{\alpha,\beta,n}\le C_{\alpha,\beta,n}<+\infty$, depending only
on~$\alpha$, $\beta$ and~$n$, such that,
for all~$x\in\R^n$, $y\in B_\varrho$ and~$\xi\in\R^n$,
\begin{equation*}
c_{\alpha,\beta,n}|\xi|\le |M(x-y, y)\xi|\le C_{\alpha,\beta,n} |\xi|.
\end{equation*}
\end{lemma}

\begin{proof} Given~$x$, $y\in\R^n$, we know that~$M(x-y, y)$ satisfies~\eqref{tfyugbhijnkml,ò},
and therefore~$M(x-y, y)$ belongs to~${\mathcal{M}}_{\alpha,\beta}$,
defined in Lemma~\ref{lemmaduewqi1234567poiuyt}.

Denoting by~$c_{\alpha,\beta,n}$ and~$C_{\alpha,\beta,n}$ the quantities in~\eqref{ALqwdf-2we-1} and~\eqref{ALqwdf-2we-2} respectively,
we know by Lemma~\ref{lemmaduewqi1234567poiuyt} that~$0<c_{\alpha,\beta,n}\le C_{\alpha,\beta,n}<+\infty$. Thus, the desired inequalities in Lemma~\ref{propnormeequiv00} follows from Lemma~\ref{lemmaduewqi1234567poiuyt}.
\end{proof}

The following is a useful perturbative estimate
regarding the interaction kernel induced by a matrix.

\begin{lemma}\label{ILQFR}
There exists~$C_\star>0$, depending only on~$n$, $\alpha$ and~$\beta$, such that,
for all~$L$, $N\in\R^{n\times n}$ for which~$ |L\xi|$, $| N\xi|\in [\beta,\alpha]$
for all~$\xi\in \mathcal{S}^{n-1}$, we have that, for all~$y\in\R^n\setminus\{0\}$,
$$ \left|\frac1{|Ly|^{n+2s}}-\frac1{|Ny|^{n+2s}}\right|\le \frac{C_\star\,\|L-N\|}{|y|^{n+2s}}.$$
\end{lemma}

\begin{proof} We claim that there exists~$C_\star>0$, depending only on~$n$, $\alpha$ and~$\beta$, such that,
for all~$L$, $N\in\R^{n\times n}$ for which~$ |L\xi|$, $ |N\xi|\in [\beta,\alpha]$
for all~$\xi\in \mathcal{S}^{n-1}$, we have that
\begin{equation}\label{0fm98Ln0698mn0-8bwnbm}
\sup_{y\in\R^n\setminus\{0\}}\left|1-\frac{|Ly|^{n+2s}}{|Ny|^{n+2s}}\right|\le C_\star\,\|L-N\|.
\end{equation}
To check this, for every~$y\in\R^n\setminus\{0\}$ and~$N$ as above, we let
$$ \R^{n\times n}\ni L\longmapsto \Phi(L):=1-\frac{|Ly|^{n+2s}}{|Ny|^{n+2s}}$$
and we observe that
\begin{eqnarray*}
\frac{\partial \Phi}{\partial L_{ij}}&=&-(n+2s)\frac{|Ly|^{n+2s-1}}{|Ny|^{n+2s}}\,\frac{\partial }{\partial L_{ij}}|Ly|\\
&=&-\frac{(n+2s)}2\frac{|Ly|^{n+2s-2}}{|Ny|^{n+2s}}\,\frac{\partial }{\partial L_{ij}}|Ly|^2\\
&=&-\frac{(n+2s)}2\frac{|Ly|^{n+2s-2}}{|Ny|^{n+2s}}\,\frac{\partial }{\partial L_{ij}}\sum_{k=1}^n \left(\sum_{m=1}^n L_{km}y_m\right)^2\\&=&-(n+2s)\frac{|Ly|^{n+2s-2}}{|Ny|^{n+2s}}\,\sum_{k,m=1}^n L_{km}y_my_j.
\end{eqnarray*}
Thus, possibly changing~$C_\star$ from line to line,
we conclude that
\begin{eqnarray*}
\left|\frac{\partial \Phi}{\partial L_{ij}}\right|\le
C_\star\,\frac{|Ly|^{n+2s-1}\,|y|}{|Ny|^{n+2s}}\le
C_\star\,\frac{|y|^{n+2s-1}\,|y|}{|y|^{n+2s}}=C_\star.
\end{eqnarray*}

Hence,
$$ \left|1-\frac{|Ly|^{n+2s}}{|Ny|^{n+2s}}\right|=|\Phi(L)-\Phi(N)|\le C_\star\|L-N\|,
$$which establishes~\eqref{0fm98Ln0698mn0-8bwnbm}.

The desired result now follows as a consequence of~\eqref{0fm98Ln0698mn0-8bwnbm}, up to renaming constants.\end{proof}

One of the byproducts of Lemma~\ref{ILQFR} is that the kernels induced by sufficiently regular nondegenerate
matrices fulfill condition~\eqref{COND:qsdfdv}.

\begin{corollary}\label{PSjmndwqoeutr98n76bv9m0v04b8n9m-39tee45ndsh76y}
Let~$\varrho\in(0,+\infty]$ and~$
M:\R^n\times B_\varrho\to\R^{n\times n}$ satisfy~\eqref{matdefpos} and let~$K$ be as in~\eqref{kerneldiv}.
Assume that~$M$ is Lipschitz continuous.

Then, condition~\eqref{COND:qsdfdv} is satisfied, and, in this situation,
$$\int_{B_1} |y|\,\big|K(x,x+y)-K(x,x-y)\big|\,dy\le C_\star\,C_M\, s,$$
where~$C_\star>0$ depends only on~$n$, $\alpha$ and~$\beta$, and~$C_M$ denotes the Lipschitz constant of~$M$.
\end{corollary}

\begin{proof} Given~$x\in\R^n$, by virtue of Lemma~\ref{ILQFR} we have that
\begin{eqnarray*}&&
\int_{B_1} |y|\,\big|K(x,x+y)-K(x,x-y)\big|\,dy\\&&\qquad=
c_{n, s}\int_{B_1} |y|\,\left|\frac{1}{ |M(x+y, -y)y|^{n+2s}}-\frac{1}{ |M(x-y, y)y|^{n+2s}}\right|\,dy\\&&\qquad\le
C_\star\,c_{n, s}\int_{B_1} |y|\,\frac{\|M(x+y, -y)-M(x-y, y)\|}{|y|^{n+2s}}\,dy\\&&\qquad\le
C_\star\,C_M\,c_{n, s}\int_{B_1} \frac{dy}{|y|^{n+2s-2}}\\&&\qquad\le C_\star\,C_M\, s
,\end{eqnarray*}
up to renaming~$C_\star$ from line to line.
\end{proof}

\section{Continuity of eigenvalues and eigenprojections}
To prove our limit results,
we leverage some perturbation results for finite-dimensional linear operators.

For a matrix~$A\in\R^{n\times n}$, we denote its eigenvalues by~$\lambda_1(A)\le\dots\le\lambda_n(A)$.
We also use the standard matrix norm
$$\|A\|:=\sup_{\xi\in {\mathcal{S}}^{n-1}}|A\xi|.$$
We recall that eigenvalues of symmetric matrices
satisfy a Lipschitz continuity, according to the following observation:

\begin{lemma}\label{WERFV}
Given two symmetric matrices~$A$, $B\in\R^{n\times n}$, we have that, for all~$i\in\{1,\dots,n\}$,
$$|\lambda_i(A)-\lambda_i(B)|\le\|A-B\|.$$
\end{lemma}

\begin{proof} Let~$C:=B-A$ and note that~$C\in\R^{n\times n}$ is symmetric. Hence, we can use formula~(4.3.16) in Corollary~4.3.15 of~\cite{MR2978290} and find that
\begin{equation}\label{odlmc213e-r3fg}
\lambda_i (A) + \lambda_1 (C) \le \lambda_i (A + C) \le \lambda_i (A) + \lambda_n (C).\end{equation}

Also, for all~$j\in\{1,\dots,n\}$ we can find~$v_j\in {\mathcal{S}}^{n-1}$ such that~$Cv_j=\lambda_j(C)v_j$ and therefore
$$ |\lambda_j(C)|=|\lambda_j(C)v_j|=|Cv_j|\le
\sup_{\xi\in {\mathcal{S}}^{n-1}}|C\xi|=\|C\|=\|A-B\|
.$$
This and~\eqref{odlmc213e-r3fg} entail that
\begin{equation*}
\lambda_i (A) -\|A-B\| \le \lambda_i (B) \le \lambda_i (A) + \|A-B\|,\end{equation*}
which yields the desired result.
\end{proof}

 The following is a restatement of a classical result about the Lipschitz continuity of eigenprojections
(see Lemma~3 in~\cite{MR1969760}; see also Lemma~4.3 in~\cite{MR1974190} and the references therein):

\begin{lemma} \label{lemmafeye542yt7jksdjklinear}
There exist~$c$, $C\in(0,+\infty)$ such that the following statement holds true.

Let~$A$, $B\in\R^{n\times n}$ be symmetric matrices, with~$\|A-B\|\le c$.
Let~$O\in\R^{n\times n}$ an orthogonal matrix such as~$O^TAO$ is diagonal with increasing diagonal entries.

Then, there exists an orthogonal matrix~$P\in\R^{n\times n}$ such as~$P^TBP$ is diagonal
with increasing diagonal entries and
$$\|O-P\|\le C\|A-B\|.$$
\end{lemma}

\section{Norm equivalency and proof of Lemma~\ref{LEMMAANCVARRO}}\label{APPE1}

Here we give the proof of Lemma~\ref{LEMMAANCVARRO}.

\begin{proof}[Proof of Lemma~\ref{LEMMAANCVARRO}] We let
$$\Omega_1:=\bigcup_{x\in\Omega}B_1(x).$$
We decompose~$\R^n$ into a family of nonoverlapping cubes~$\{Q_j\}_{j\in\N}$ with faces normal to the coordinate axes and sides of length~$\ell:=\frac{\varrho}{8\sqrt{n}}$. We label these cubes in such a way that
\begin{equation}\label{OMEDEFGH}\Omega_1\subseteq\bigcup_{j=0}^N Q_j,\end{equation}
with~$N\in\N$ depending only on~$n$, $\varrho$ and~$\Omega$.

We also denote by~$Q_j^\star$ the union of~$Q_j$ and its adjacent cubes. In this way, $Q_j^\star$ is a cube with the same center as~$Q_j$ and sides of length~$3\ell$.

In this framework, we see that
\begin{equation}\label{asdcrStY-010}\begin{split}
&\iint_{\R^n\times\R^n}|u(x)-u(z)|^2 K(x,z)\,dx\,dz\ge\sum_{j=0}^N
\iint_{Q_j\times\R^n}|u(x)-u(z)|^2 K(x,z)\,dx\,dz\\&\quad
=\sum_{{0\le j\le N}\atop{i\in\N}}
\iint_{Q_j\times Q_i}|u(x)-u(z)|^2 K(x,z)\,dx\,dz
\ge\sum_{ j=0}^{ N}
\iint_{Q_j\times Q_j^\star}|u(x)-u(z)|^2 K(x,z)\,dx\,dz
.\end{split}
\end{equation}

Now, if~$x\in Q_j$ and~$z\in Q_j^\star$, we have
that~$|x-z|\le 3\sqrt{n}\ell<\varrho$ and accordingly,
by~\eqref{QrehtgDWFfbg0jolwfdv:Qwfgb}, we have that
$$K(x,z)\ge\frac{c}{|x-z|^{n+2s}}.$$
This and~\eqref{asdcrStY-010} give that
\begin{equation}\label{asdcrStY-011}
\iint_{\R^n\times\R^n}|u(x)-u(z)|^2 K(x,z)\,dx\,dz\ge
c\sum_{ j=0}^{ N}
\iint_{Q_j\times Q_j^\star}\frac{|u(x)-u(z)|^2 }{|x-z|^{n+2s}}\,dx\,dz.\end{equation}

Moreover, by virtue of~\eqref{OMEDEFGH},
\begin{equation}\label{asdcrStY-011a}
\begin{split}&\iint_{\Omega_1\times\Omega_1}\frac{|u(x)-u(z)|^2}{|x-z|^{n+2s}}\,dx\,dz
\le\sum_{j=0}^N\iint_{Q_j\times\Omega_1}\frac{|u(x)-u(z)|^2}{|x-z|^{n+2s}}\,dx\,dz\\&\quad\le
\sum_{j=0}^N\iint_{{Q_j\times\Omega_1}\atop{\{|x-z|<\ell\}}}\frac{|u(x)-u(z)|^2}{|x-z|^{n+2s}}\,dx\,dz
+\sum_{j=0}^N\iint_{{Q_j\times\Omega_1}\atop{\{|x-z|\ge\ell\}}}\frac{\big(|u(x)|+|u(z)|\big)^2}{\ell^{n+2s}}\,dx\,dz\\&\quad\le
\sum_{j=0}^N\iint_{{Q_j\times Q_j^\star}}\frac{|u(x)-u(z)|^2}{|x-z|^{n+2s}}\,dx\,dz
+\widetilde C\,\|u\|_{L^2(\R^n)}^2,
\end{split}
\end{equation}
where we denote by~$\widetilde{C}>0$ an opportune constant, depending only on~$n$, $s$, $\rho$ and~$\Omega$, and possibly varying from line to line.

Moreover,
\begin{equation}\label{asdcrStY-012}
\begin{split}&
\iint_{\Omega\times(\R^n\setminus\Omega_1)}\frac{|u(x)-u(z)|^2}{|x-z|^{n+2s}}\,dx\,dz
\le\iint_{{\R^n\times\R^n}\atop{\{|x-z|\ge1\}}}\frac{|u(x)-u(z)|^2}{|x-z|^{n+2s}}\,dx\,dz\\
&\qquad\le 
\iint_{{\R^n\times\R^n}\atop{\{|x-z|\ge1\}}}\frac{\big(|u(x)|+|u(z)|\big)^2}{|x-z|^{n+2s}}\,dx\,dz \le
\widetilde{C}\|u\|_{L^2(\R^n)}^2.\end{split}
\end{equation}

Now we claim that
\begin{equation}\label{IQSDSFDVQWFGQwefgr} \|u\|_{L^2(\R^n)}^2 \le
\widetilde{C}\iint_{{\R^n\times\R^n}\atop{\{|x-y|\le\varrho\}}}\frac{|u(x)-u(z)|^2}{|x-z|^{n+2s}}\,dx\,dz.\end{equation}
To check this, we argue by contradiction and suppose that there exists a sequence of functions~$u_m$ such that $u_m=0$ in $\R^n\setminus\Omega$ and 
$$ 1= \|u_m\|_{L^2(\R^n)}^2 \ge
m\iint_{{\R^n\times\R^n}\atop{\{|x-y|\le\varrho\}}}\frac{|u_m(x)-u_m(z)|^2}{|x-z|^{n+2s}}\,dx\,dz.$$
In particular, for each~$j\in\{0,\dots,N\}$,
\begin{equation}\label{w-2eirfegjpbl}\frac1m\ge\iint_{{Q_j\times\R^n}\atop{\{|x-y|\le\varrho\}}}\frac{|u_m(x)-u_m(z)|^2}{|x-z|^{n+2s}}\,dx\,dz\ge\iint_{{Q_j\times Q_{j,\varrho}}}\frac{|u_m(x)-u_m(z)|^2}{|x-z|^{n+2s}}\,dx\,dz,\end{equation}
where
$$Q_{j,\varrho}:=\bigcup_{x\in Q_j}B_\varrho(x).$$
This gives that
$$ \iint_{{Q_j\times Q_j}}\frac{|u_m(x)-u_m(z)|^2}{|x-z|^{n+2s}}\,dx\,dz+\|u_m\|_{L^2(Q_j)}\le\frac1m+1\le2$$
and therefore, by compactness
(see e.g.~\cite[Theorem~7.1]{MR2944369}), up to a relabeled subsequence, we can suppose that~$u_m$
converges to some~$u_\infty$ in~$L^2(Q_j)$ and a.e. in~$Q_j$, for all~$j\in\{0,\dots,N\}$.

As a result,
\begin{equation}\label{Qwdef0jgrotbnglrtyh}
1=\lim_{m\to+\infty} \|u_m\|_{L^2(\R^n)}^2=
\lim_{m\to+\infty} \sum_{j=0}^N\|u_m\|_{L^2(Q_j)}^2=\sum_{j=0}^N\|u_\infty\|_{L^2(Q_j)}^2=
\|u_\infty\|_{L^2(\Omega)}^2.
\end{equation}

However, it follows from~\eqref{w-2eirfegjpbl} and Fatou's Lemma that, for all~$j\in\{0,\dots,N\}$,
$$ \iint_{{Q_j\times Q_{j,\varrho}}}\frac{|u_\infty(x)-u_\infty(z)|^2}{|x-z|^{n+2s}}\,dx\,dz=0.$$
This entails that~$u_\infty$ is constant in each~$Q_j$ and actually in~$Q_{j,\varrho}$.
As a result, $u_\infty$ is constant, and therefore constantly equal to zero (since~$u_\infty=0$ outside~$\Omega$). But this is in contradiction with~\eqref{Qwdef0jgrotbnglrtyh} and the proof of~\eqref{IQSDSFDVQWFGQwefgr} is complete.

Gathering~\eqref{asdcrStY-011}, \eqref{asdcrStY-011a},
\eqref{asdcrStY-012} and~\eqref{IQSDSFDVQWFGQwefgr}, we conclude that
\begin{eqnarray*}&&
\iint_{\R^n\times\R^n} \frac{|u(x)-u(z)|^2}{|x-z|^{n+2s}}\,dx\,dz\\&=&
\iint_{\Omega_1\times\Omega_1}\frac{|u(x)-u(z)|^2}{|x-z|^{n+2s}}\,dx\,dz+
2\iint_{\Omega_1\times(\R^n\setminus\Omega_1)}\frac{|u(x)-u(z)|^2}{|x-z|^{n+2s}}\,dx\,dz
\\&=&\iint_{\Omega_1\times\Omega_1}\frac{|u(x)-u(z)|^2}{|x-z|^{n+2s}}\,dx\,dz
+2\iint_{\Omega\times(\R^n\setminus\Omega_1)}\frac{|u(x)-u(z)|^2}{|x-z|^{n+2s}}\,dx\,dz
\\&&\qquad+2\iint_{(\Omega_1\setminus\Omega)\times(\R^n\setminus\Omega_1)}\frac{|u(x)-u(z)|^2}{|x-z|^{n+2s}}\,dx\,dz
\\&\le&\iint_{\Omega_1\times\Omega_1}\frac{|u(x)-u(z)|^2}{|x-z|^{n+2s}}\,dx\,dz
+\widetilde C\|u\|^2_{L^2(\R^n)}
+0\\&
\le &\sum_{j=0}^N\iint_{{Q_j\times Q_j^\star}}\frac{|u(x)-u(z)|^2}{|x-z|^{n+2s}}\,dx\,dz+\widetilde C\|u\|^2_{L^2(\R^n)}\\
&\le&\frac1c\iint_{\R^n\times\R^n}|u(x)-u(z)|^2 K(x,z)\,dx\,dz+\widetilde C\|u\|_{L^2(\R^n)}^2\\
&\le&\frac1c\iint_{\R^n\times\R^n}|u(x)-u(z)|^2 K(x,z)\,dx\,dz+
\widetilde{C}\iint_{{\R^n\times\R^n}\atop{\{|x-y|\le\varrho\}}}\frac{|u(x)-u(z)|^2}{|x-z|^{n+2s}}\,dx\,dz\\
&\le&\frac{\widetilde{C}}c\iint_{\R^n\times\R^n}|u(x)-u(z)|^2 K(x,z)\,dx\,dz,
\end{eqnarray*}as desired.
\end{proof}

\section{Asymptotics as $s\searrow 0$ and $s\nearrow 1$}\label{SEC2}

In this appendix, we provide the details of the proofs of the asymptotic statements
in Proposition~\ref{Proplimito1}.

In the following computations the asymptotic behaviour of the constant~$c_{n,s}$ will play an important role. For this reason, we recall that
\begin{equation}\label{proplimiticNs}
\lim_{s\searrow 0} \frac{c_{n, s}}{s} = \frac{2}{\omega_{n-1}} \quad\mbox{ and }\quad \lim_{s\nearrow 1} \frac{c_{n, s}}{1-s} = \frac{4n}{\omega_{n-1}},
\end{equation}
see e.g. formula~(1.50) in~\cite{GENTLE}.

\begin{proof}[Proof of Proposition~\ref{Proplimito1}]
Without loss of generality, we can suppose that~$\varrho=+\infty$, since,
for all~$\varrho_0\in(0,+\infty)$, we have that
\begin{eqnarray*}
&&\lim_{s\nearrow 1}\frac{c_{n, s}}{2}\iint_{\R^n\times (\R^n\setminus B_{\varrho_0})} \frac{|u(x)-u(x-y)| |\varphi(x)-\varphi(x-y)|}{|M(x-y, y) y|^{n+2s}} \,dx\, dy\\&&\qquad \le 
\lim_{s\nearrow 1} C(1-s) \iint_{\R^n\times (\R^n\setminus B_{\varrho_0})} \frac{|u(x)-u(x-y)| |\varphi(x)-\varphi(x-y)|}{|y|^{n+2s}} \,dx\, dy\\&&\qquad
\le \lim_{s\nearrow 1} C(1-s)
\iint_{\R^n\times (\R^n\setminus B_{\varrho_0})} \frac{(u(x)-u(x-y))^2+(\varphi(x)-\varphi(x-y))^2}{|y|^{n+2s}} \,dx\, dy\\&&\qquad\le
\lim_{s\nearrow 1} C(1-s)
\int_{\R^n\setminus B_{\varrho_0}} \frac{\|u\|_{L^2(\R^n)}^2 +\|\varphi\|_{L^2(\R^n)}^2}{|y|^{n+2s}} \,dy\\&&\qquad\le
\lim_{s\nearrow 1} C(1-s)\Big(\|u\|_{L^2(\R^n)}^2 +\|\varphi\|_{L^2(\R^n)}^2\Big)\\&&\qquad=0.
\end{eqnarray*}

Now, we remark that
\begin{equation}\label{SIMP1}
{\mbox{it suffices to prove~\eqref{limitesto1distr} for all~$u\in H^1(\R^n)$ which are compactly supported.}}
\end{equation}
Indeed, given~$u\in H^1(\R^n)$ and~$R>0$, we pick~$\tau_R\in C^\infty_c(B_{R+1})$ with~$\tau_R=1$ in~$B_R$. We set~$u_R:=\tau_R u$ and~$v_R:=u-u_R=(1-\tau_R)u$. Hence, if~\eqref{limitesto1distr} holds true
for all~$u\in H^1(\R^n)$ which are compactly supported, we have that,
given~$\varphi\in C^\infty_c(B_{R_0})$,
\begin{equation*}\begin{split}&
\lim_{s\nearrow 1}\frac{c_{n, s}}{2}\iint_{\R^n\times\R^n} \frac{(u_R(x)-u_R(x-y))(\varphi(x)-\varphi(x-y))}{|M(x-y, y) y|^{n+2s}} \,dx\, dy\\&\qquad= \sum_{i, j=1}^n \,\int_{\R^n} A_{ij}(x)\frac{\partial u_R}{\partial x_i}(x) \frac{\partial\varphi}{\partial x_j}(x)\, dx\end{split}
\end{equation*}
and consequently, keeping in mind~\eqref{matdefpos} and assuming~$R>R_0$, {\footnotesize
\begin{eqnarray*}&&
\left|\lim_{s\nearrow 1}\frac{c_{n, s}}{2}\iint_{\R^n\times\R^n} \frac{(u(x)-u(x-y))(\varphi(x)-\varphi(x-y))}{|M(x-y, y) y|^{n+2s}} \,dx\, dy- \sum_{i, j=1}^n \,\int_{\R^n} A_{ij}(x)\frac{\partial u}{\partial x_i}(x) \frac{\partial\varphi}{\partial x_j}(x)\, dx
\right|\\&&\quad=
\left|\lim_{s\nearrow 1}\frac{c_{n, s}}{2}\iint_{\R^n\times\R^n} \frac{(v_R(x)-v_R(x-y))(\varphi(x)-\varphi(x-y))}{|M(x-y, y) y|^{n+2s}} \,dx\, dy- \sum_{i, j=1}^n \,\int_{\R^n} A_{ij}(x)\frac{\partial v_R}{\partial x_i}(x) \frac{\partial\varphi}{\partial x_j}(x)\, dx
\right|\\&&\quad\le
\lim_{s\nearrow 1} C(1-s)\iint_{\R^n\times\R^n} \frac{|v_R(x)-v_R(x-y)|\,|\varphi(x)-\varphi(x-y)|}{|y|^{n+2s}} \,dx\, dy
+C\int_{B_{R_0}} |\nabla v_R(x)|\, dx\\&&\quad=
\lim_{s\nearrow 1} C(1-s)\iint_{\R^n\times\R^n} \frac{|v_R(x)-v_R(z)|\,|\varphi(x)-\varphi(z)|}{|x-z|^{n+2s}} \,dx\, dz
+0\\&&\quad\le
\lim_{s\nearrow 1} C(1-s)\iint_{\R^n\times B_{R_0}} \frac{|v_R(x)-v_R(z)|\,|\varphi(x)-\varphi(z)|}{|x-z|^{n+2s}} \,dx\, dz\\&&\quad=
\lim_{s\nearrow 1} C(1-s)\iint_{\R^n\times B_{R_0}} \frac{|v_R(x)|\,|\varphi(x)-\varphi(z)|}{|x-z|^{n+2s}} \,dx\, dz\\&&\quad=
\lim_{s\nearrow 1} C(1-s)\iint_{(\R^n\setminus B_R)\times B_{R_0}} \frac{|v_R(x)|\,|\varphi(x)-\varphi(z)|}{|x-z|^{n+2s}} \,dx\, dz\\&&\quad=
\lim_{s\nearrow 1} C(1-s)\iint_{(\R^n\setminus B_R)\times B_{R_0}} \frac{(1-\tau_R(x))|u(x)|\,|\varphi(z)|}{|x-z|^{n+2s}} \,dx\, dz\\&&\quad\le\lim_{s\nearrow 1} C(1-s)\int_{\R^n\setminus B_R} \frac{|u(x)|}{|x|^{n+2s}} \,dx\\&&\quad
\le\lim_{s\nearrow 1} C(1-s)\sqrt{\int_{\R^n\setminus B_R} |u(x)|^2 \,dx\,
\int_{\R^n\setminus B_R} \frac{dx}{|x|^{2n+4s}} }\\&&\quad=0
\end{eqnarray*}}

\noindent for some~$C>0$, independent of~$s$ and~$R$ and possibly varying from line to line.
The proof of~\eqref{SIMP1} is thereby complete.

Now we point out that
\begin{equation}\label{SIMP2}
{\mbox{it suffices to prove~\eqref{limitesto1distr} for all~$u\in C^\infty_c(\R^n)$.}}
\end{equation}
To check this, in light of~\eqref{SIMP1}, we can focus on the case in which~$u\in H^1(\R^n)$ with support in~$B_{R_0}$
and~$\varphi\in C^\infty_c(B_{R_0})$, for some~$R_0>0$. Thus, we take~$u_\e\in C^\infty_c(B_{R_0})$
with~$[u-u_\e]_1\le\e$ and define~$v_\e:=u-u_\e$. In this way,  if~\eqref{limitesto1distr} hods true for~$u_\e$ it follows that
\begin{equation*}\begin{split}&
\lim_{s\nearrow 1}\frac{c_{n, s}}{2}\iint_{\R^n\times\R^n} \frac{(u_\e(x)-u_\e(x-y))(\varphi(x)-\varphi(x-y))}{|M(x-y, y) y|^{n+2s}} \,dx\, dy\\&\qquad= \sum_{i, j=1}^n \,\int_{\R^n} A_{ij}(x)\frac{\partial u_\e}{\partial x_i}(x) \frac{\partial\varphi}{\partial x_j}(x)\, dx\end{split}
\end{equation*}
and therefore, employing~\cite[Lemma~2.1]{MR4736013},
{\footnotesize
\begin{eqnarray*}&&
\left|\lim_{s\nearrow 1}\frac{c_{n, s}}{2}\iint_{\R^n\times\R^n} \frac{(u(x)-u(x-y))(\varphi(x)-\varphi(x-y))}{|M(x-y, y) y|^{n+2s}} \,dx\, dy- \sum_{i, j=1}^n \,\int_{\R^n} A_{ij}(x)\frac{\partial u}{\partial x_i}(x) \frac{\partial\varphi}{\partial x_j}(x)\, dx
\right|\\&&\quad=\left|\lim_{s\nearrow 1}\frac{c_{n, s}}{2}\iint_{\R^n\times\R^n} \frac{(v_\e(x)-v_\e(x-y))(\varphi(x)-\varphi(x-y))}{|M(x-y, y) y|^{n+2s}} \,dx\, dy- \sum_{i, j=1}^n \,\int_{\R^n} A_{ij}(x)\frac{\partial v_\e}{\partial x_i}(x) \frac{\partial\varphi}{\partial x_j}(x)\, dx
\right|\\&&\quad\le
\lim_{s\nearrow 1}C(1-s)\iint_{\R^n\times\R^n} \frac{|v_\e(x)-v_\e(x-y)|\,|\varphi(x)-\varphi(x-y)|}{|y|^{n+2s}} \,dx\, dy
+ C \int_{B_{R_0}} |\nabla v_\e(x)|\, dx\\&&\quad\le
\lim_{s\nearrow 1}C[v_\e]_s\,[\varphi]_s
+C[v_\e]_1\\&&\quad\le
\lim_{s\nearrow 1}C[v_\e]_1\,[\varphi]_1
+C[v_\e]_1\\&&\quad\le C\e,
\end{eqnarray*}}

\noindent for some~$C>0$, independent of~$s$ and~$\e$ and possibly varying from line to line.
Accordingly, the claim in~\eqref{SIMP2} follows by sending~$\e\searrow0$.

Now we focus on the proof of~\eqref{limitesto1distr}
and, thanks to~\eqref{SIMP2},
we will assume that~$u$, $\varphi\in C^\infty_c(B_{R_0})$, for some~$R_0>0$.

We consider a mollification~$M_\e$ of~$M$ and 
we remark that, by denoting a standard even mollifier by~$\eta_\e\in C^\infty_c(B_\e)$, for all~$\xi\in{\mathcal{S}}^{n-1}$,
\begin{equation}\label{MALDoiwfnoiy504ojh} M_\e(x-y, y)\xi\cdot\xi=\iint_{\R^n\times\R^n}
M(x-y-w, y-z)\xi\cdot\xi\,\eta_\e(w)\,\eta_\e(z)\,dw\,dz\in[\beta,\alpha]
,\end{equation}
due to~\eqref{matdefpos}.

We now let~$r$ be as in~\eqref{ASSUNZ:MM} and (thanks to~\cite[Theorem~9.8]{MR3381284}) conclude that, for every~$R>0$, as~$\e\searrow0$,
\begin{equation}\label{qas0jodv039ertuig34-9rtog}
{\mbox{$M_\e$ converges to~$M$ uniformly in~$B_R\times B_{r/2}$.}}\end{equation}

Furthermore, by~\cite[equation~(3.7) and Appendix~W]{MR3967804},
\begin{equation}\label{PAKJMDLFjmwpluyiKJS9i0d27j35v6gbnbv4fUmOSgo68}
\begin{split}&
\lim_{s\nearrow 1}\frac{c_{n, s}}{2}\iint_{\R^n\times\R^n} \frac{(u(x)-u(x-y))(\varphi(x)-\varphi(x-y))}{|M_\e(x-y, y) y|^{n+2s}} \,dx\, dy\\&\qquad= \sum_{i, j=1}^n \,\int_{\R^n} A_{ij,\e}(x)\frac{\partial u}{\partial x_i}(x) \frac{\partial\varphi}{\partial x_j}(x)\, dx,\end{split}
\end{equation}
where
\begin{equation}\label{PAKJMDLFjmwpluyiKJS9i0d27j35v6gbnbv4fUmOSgo682}A_{ij,\e}(x) := \frac{n}{\omega_{n-1}}\int_{\mathcal S^{n-1}} \frac{\psi_i \psi_j}{|M_\e(x, 0)\,\psi|^{n+2}} \,d\mathcal H_\psi^{n-1}.\end{equation}

This and~\eqref{matdefpos} yield that
\begin{eqnarray*}&&
\lim_{s\nearrow 1}\frac{c_{n, s}}{2}\left|\;\iint_{\R^n\times(\R^n\setminus B_{r/2})} \frac{(u(x)-u(x-y))(\varphi(x)-\varphi(x-y))}{|M_\e(x-y, y) y|^{n+2s}} \,dx\, dy\right|\\&&\quad\le
\lim_{s\nearrow 1}C(1-s)\iint_{\R^n\times(\R^n\setminus B_{r/2})} \frac{|u(x)-u(x-y)|\,|\varphi(x)-\varphi(x-y)|}{| y|^{n+2s}} \,dx\, dy
\\&&\quad\le
\lim_{s\nearrow 1}C(1-s)\left(\,\iint_{\R^n\times(\R^n\setminus B_{r/2})} \frac{|u(x)-u(x-y)|^2}{| y|^{n+2s}} \,dx\, dy\right)^{\frac12}\times\\
&&\qquad\qquad\times\left(\,\iint_{\R^n\times(\R^n\setminus B_{r/2})} \frac{|\varphi(x)-\varphi(x-y)|^2}{| y|^{n+2s}} \,dx\, dy\right)^{\frac12}
\\&&\quad\le\lim_{s\nearrow 1} C(1-s) \int_{\R^n\setminus B_{r/2}}\frac{dy}{|y|^{n+2s}}\\
&&\quad\le\lim_{s\nearrow 1}\frac{C(1-s)}{s\,r^{2s}}\\&&\quad=0,
\end{eqnarray*}
for some~$C>0$ independent of~$r$, $s$ and~$\e$.

As a consequence, we rewrite~\eqref{PAKJMDLFjmwpluyiKJS9i0d27j35v6gbnbv4fUmOSgo68} as
\begin{equation*}
\begin{split}&
\lim_{s\nearrow 1}\frac{c_{n, s}}{2}\iint_{\R^n\times B_{r/2}} \frac{(u(x)-u(x-y))(\varphi(x)-\varphi(x-y))}{|M_\e(x-y, y) y|^{n+2s}} \,dx\, dy\\&\qquad= \sum_{i, j=1}^n \,\int_{\R^n} A_{ij,\e}(x)\frac{\partial u}{\partial x_i}(x) \frac{\partial\varphi}{\partial x_j}(x)\, dx\end{split}
\end{equation*}
and therefore, in light of~\eqref{qas0jodv039ertuig34-9rtog}
and~\eqref{PAKJMDLFjmwpluyiKJS9i0d27j35v6gbnbv4fUmOSgo682},
\begin{equation}\label{sdcjovbw90tyhuj-0oylukwferfbg}
\begin{split}&
\lim_{\e\searrow0}\lim_{s\nearrow 1}\frac{c_{n, s}}{2}\iint_{\R^n\times B_{r/2}} \frac{(u(x)-u(x-y))(\varphi(x)-\varphi(x-y))}{|M_\e(x-y, y) y|^{n+2s}} \,dx\, dy\\&\qquad= \sum_{i, j=1}^n \,\int_{\R^n} A_{ij}(x)\frac{\partial u}{\partial x_i}(x) \frac{\partial\varphi}{\partial x_j}(x)\, dx.\end{split}
\end{equation}

Using~\eqref{MALDoiwfnoiy504ojh} and Lemma~\ref{ILQFR}, we obtain that
\begin{eqnarray*}&&
\left| \frac{1}{|M(x-y, y) y|^{n+2s}}-\frac{1}{|M_\e(x-y, y) y|^{n+2s}}\right|
\\&&\qquad\le
\frac{C_\star\,\big\|M(x-y, y)-M_\e(x-y, y) \big\|}{|y|^{n+2s}}.
\end{eqnarray*}

Accordingly,
up to renaming constants we have that {\footnotesize
\begin{eqnarray*}&&
\frac{c_{n, s}}{2}\iint_{\R^n\times B_{r/2}} 
|u(x)-u(x-y)|\,|\varphi(x)-\varphi(x-y)|
\left|\frac{1}{|M(x-y, y) y|^{n+2s}} -\frac{1}{|M_\e(x-y, y) y|^{n+2s}}\right|\,dx \,dy\\&&\quad\le
C_\star \,(1-s)\iint_{\R^n\times B_{r/2}} \frac{
|u(x)-u(x-y)|\,|\varphi(x)-\varphi(x-y)|\,\big\|M(x-y, y)-M_\e(x-y, y) \big\|}{|y|^{n+2s}}\,dx\,dy
\\&&\quad=
C_\star \,(1-s)\iint_{B_{R_0+r}\times B_{r/2}} \frac{
|u(x)-u(x-y)|\,|\varphi(x)-\varphi(x-y)|\,\big\|M(x-y, y)-M_\e(x-y, y) \big\|}{|y|^{n+2s}}\,dx\,dy\\&&\quad\le
C_\star \,(1-s)\,\|M-M_\e\|_{L^\infty( B_{R_0+2r}\times B_{r/2})}\,
\iint_{B_{R_0+r}\times B_{r/2}} |y|^{2-n-2s}\;
 \,dx\,dy\\
&&\quad\le C_\star \,\|M-M_\e\|_{L^\infty( B_{R_0+2r}\times B_{r/2})}
,\end{eqnarray*}}

\noindent and this is infinitesimal as~$\e\searrow0$,
owing to~\eqref{qas0jodv039ertuig34-9rtog}.

{F}rom this and~\eqref{sdcjovbw90tyhuj-0oylukwferfbg} we obtain~\eqref{limitesto1distr}, as desired.
\end{proof}

For completeness, we also provide here the asymptotic of the operator~${\mathcal{L}}_{M,\varrho, s}$ as~$s\searrow0$.

\begin{proposition}\label{Proplimito0}
Let $\overline{s}\in (0, 1)$ and $u\in H^{\overline{s}}(\R^n)$. Let~$\varrho\in(0,+\infty]$ and~$M:\R^n\times B_\varrho\to\R^{n\times n}$ be a matrix-valued function with measurable entries that satisfy~\eqref{matdefpos} and~\eqref{structuralM}.

If~$\varrho=+\infty$, assume also that
\begin{equation}\label{Minfty}
\begin{split}
&\mbox{for any $\psi\in\mathcal S^{n-1}$, for any $x\in\R^n$ and for any $i, j\in\{1,\dots, n\}$ there exists}\\
&M^\infty_{ij}(x, \psi):=\lim_{t\to +\infty} M_{ij}(x-t\psi, t\psi)<+\infty.
\end{split}
\end{equation}

Then, for any $\varphi\in C^\infty_c(\R^n)$,
\begin{equation}\label{limitesto0distr}\begin{split}&
\lim_{s\searrow 0}  
\frac{c_{n, s}}{2}\iint_{\R^n\times B_\varrho} \frac{(u(x)-u(x-y))(\varphi(x)-\varphi(x-y))}{|M(x-y, y) y|^{n+2s}} \,dx\, dy\\&\qquad
= \begin{dcases}\displaystyle
\frac{1}{\omega_{n-1}} \int_{\R^n}u(x)\varphi(x) \int_{\mathcal S^{n-1}}\frac{d\mathcal H^{n-1}_\psi}{|M^\infty(x, \psi)\,\psi|^n} \,dx 
&{\mbox{ if }} \varrho=+\infty,\\
0 &{\mbox{ if }} \varrho\in(0,+\infty). \end{dcases}
\end{split}\end{equation}
\end{proposition}

To establish Proposition~\ref{Proplimito0}, we will make use of the following observations:
 
\begin{lemma}\label{lemmamatrice}
Let~$M:\R^n\times\R^n\to\R^{n\times n}$ be a measurable matrix-valued function
satisfying~\eqref{Minfty} for any~$x\in\R^n$. Let~$R>0$.

Then,
\[
\lim_{s\searrow 0} c_{n, s}\int_{\R^n\setminus B_R} \frac{dy}{|M(x-y, y)\, y|^{n+2s}} = \frac{1}{\omega_{n-1}} \int_{\mathcal S^{n-1}}\frac{d\mathcal H^{n-1}_\psi}{|M^\infty(x, \psi)\,\psi|^n}.
\]
\end{lemma}

\begin{proof}
We observe that
\[
\int_{\R^n\setminus B_R} \frac{dy}{|M(x-y, y)\, y|^{n+2s}} = \int_R^{+\infty}
\int_{\mathcal S^{n-1}}\frac{ d\rho\,d\mathcal H^{n-1}_\psi}{\rho^{1+2s}|M(x-\rho\psi, \rho\psi)\,\psi|^{n+2s}}.
\]
Now, by using the change of variables $r:=(\rho/R)^{2s}$, we see that
\[
\int_{\R^n\setminus B_R} \frac{dy}{|M(x-y, y)\, y|^{n+2s}} = \frac{1}{2sR^{2s}}\int_1^{+\infty}\int_{\mathcal S^{n-1}}\frac{dr\,d\mathcal H^{n-1}_\psi}{r^2|M(x-r^{\frac{1}{2s}}\psi, r^{\frac{1}{2s}}\psi)\,\psi|^{n+2s}}.
\]

Thus, by~\eqref{matdefpos},  we get
\[
\frac{1}{r^2|M(x-r^{\frac{1}{2s}}\psi, r^{\frac{1}{2s}}\psi)\,\psi|^{n+2s}}\le \frac{1}{r^2 c_{\alpha,\beta,n}^{n+2s}}\in L^1\Big((1,+\infty)\times\mathcal S^{n-1}; dr\,d\mathcal H^{n-1}_\psi\Big)
\]
Accordingly, by the Dominated Convergence Theorem, \eqref{proplimiticNs} and~\eqref{Minfty}, we deduce that
\[
\begin{split}
\lim_{s\searrow 0} c_{n, s}&\int_{\R^n\setminus B_R} \frac{dy}{|M(x-y, y)\, y|^{n+2s}} = \lim_{s\searrow 0} \frac{ c_{n, s}}{2sR^{2s}}\int_1^{+\infty}\int_{\mathcal S^{n-1}}\frac{dr\,d\mathcal H^{n-1}_\psi}{r^2|M(x-r^{\frac{1}{2s}}\psi, r^{\frac{1}{2s}}\psi)\,\psi|^{n+2s}}\\
&= \frac{1}{\omega_{n-1}}\int_1^{+\infty}\frac{dr}{r^2} \int_{\mathcal S^{n-1}}\frac{d\mathcal H^{n-1}_\psi}{|M^\infty(x, \psi)\,\psi|^n} \\
&= \frac{1}{\omega_{n-1}} \int_{\mathcal S^{n-1}}\frac{d\mathcal H^{n-1}_\psi}{|M^\infty(x, \psi)\,\psi|^n},
\end{split}
\]
as desired.
\end{proof}

\begin{lemma}\label{lemmapercheufbvbv09876}
Let $\overline{s}\in (0, 1)$ and $u$, $\varphi\in H^{\overline{s}}(\R^n)$. 
Let~$\varrho\in(0,+\infty]$ and~$M:\R^n\times B_\varrho\to\R^{n\times n}$ be a measurable matrix-valued function. Let~$R\in(0,\varrho)$.

Then,
\[
\lim_{s\searrow 0} c_{n, s}\iint_{\R^n\times B_R} \frac{(u(x)-u(x-y))(\varphi(x)-\varphi(x-y))}{|M(x-y, y) y|^{n+2s}} \,dx\, dy=0.
\]
\end{lemma}

\begin{proof}
Up to scaling, we can suppose that~$R=1$.
By~\eqref{matdefpos}, we infer that
\begin{equation}\label{serdftgybhu44}\begin{split}&
\left|c_{n, s}\iint_{\R^n\times B_1} \frac{(u(x)-u(x-y))(\varphi(x)-\varphi(x-y))}{|M(x-y, y) y|^{n+2s}} \,dx\, dy\right| \\&\qquad \le\frac{c_{n, s}}{\beta^{n+2s}} \int_{\R^n}\left(\,\int_{B_1} \frac{|u(x)-u(x-y)||\varphi(x)-\varphi(x-y)|}{|y|^{n+2s}} \,dy\right) \,dx.\end{split}
\end{equation}
We set
\[
g(s):=\int_{\R^n}\left(\,\int_{B_1} \frac{|u(x)-u(x-y)||\varphi(x)-\varphi(x-y)|}{|y|^{n+2s}} \,dy\right) \,dx.
\]
We notice that $g(s)$ is nondecreasing and  that~$g(\overline{s})<+\infty$.
{F}rom this, we infer that
\[
\lim_{s\searrow 0} g(s) \le g(\overline{s}).
\]
Using this information, \eqref{proplimiticNs} and~\eqref{serdftgybhu44}, we obtain the desired result.
\end{proof}

\begin{proof}[Proof of Proposition~\ref{Proplimito0}]
If~$\varrho\in(0,+\infty)$, the desired limit follows from Lemma~\ref{lemmapercheufbvbv09876}. Hence, from now on, we suppose that~$\varrho=+\infty$.

We observe that
\begin{equation}\label{ygiuhoijnpmkoèlp,ò22}
\begin{split}
&\frac{c_{n, s}}{2}\iint_{\R^n\times \R^n} \frac{(u(x)-u(x-y))(\varphi(x)-\varphi(x-y))}{|M(x-y, y) y|^{n+2s}} \,dx\, dy\\
&\quad=\frac{c_{n, s}}{2}\int_{\R^n}\left(\,\int_{B_1} \frac{(u(x)-u(x-y))(\varphi(x)-\varphi(x-y))}{|M(x-y, y) y|^{n+2s}} \,dy\right) \,dx\\
&\qquad+ \frac{c_{n, s}}{2}\int_{\R^n}\left(\,\int_{\R^n\setminus B_1} \frac{(u(x)-u(x-y))(\varphi(x)-\varphi(x-y))}{|M(x-y, y) y|^{n+2s}} \,dy\right) \,dx\\&\quad=: I_1 +I_2.
\end{split}
\end{equation}

In light of Lemma~\ref{lemmapercheufbvbv09876}, we have that
\begin{equation}\label{limites000I1}
\lim_{s\searrow 0} |I_1|=0.
\end{equation}

We now focus on $I_2$ and we claim that
\begin{equation}\label{claimcrtfyvgubhijnòom}
\lim_{s\searrow 0} \frac{c_{n, s}}{2} \int_{\R^n}\left(\,\int_{\R^n\setminus B_1} \frac{u(x-y)\varphi(x)}{|M(x-y, y) y|^{n+2s}} \,dy\right) \,dx = 0.
\end{equation}
Indeed, by~\eqref{matdefpos} and the H\"older inequality, we have
\[
\begin{split}
&\int_{\R^n}\left(\,\int_{\R^n\setminus B_1} \frac{|u(x-y)| |\varphi(x)|}{|M(x-y, y) y|^{n+2s}}\, dy\right) \,dx\le\frac{1}{\beta^{n+2s}}\int_{\R^n} |\varphi(x)|\left(\,\int_{\R^n\setminus B_1} \frac{|u(x-y)|}{|y|^{n+2s}} \,dy\right) \,dx\\
&\quad\le\frac{1}{\beta^{n+2s}}\int_{\R^n} |\varphi(x)| \left(\,\int_{\R^n\setminus B_1(x)} |u(z)|^2 dz\right)^\frac12  \left(\,\int_{\R^n\setminus B_1} \frac{dy}{|y|^{2n+4s}}\right)^\frac12 \,dx\\
&\quad\le\frac{1}{\beta^{n+2s}}\left(\frac{\omega_{n-1}}{n+4s}\right)^{\frac12} \|u\|_{L^2(\R^n)} \|\varphi\|_{L^1(\R^n)}.
\end{split}
\]
{F}rom this and~\eqref{proplimiticNs}, we deduce that
\[
\begin{split}
&\lim_{s\searrow 0} \frac{c_{n, s}}{2} \int_{\R^n}\left(\,\int_{\R^n\setminus B_1} \frac{u(x-y)\varphi(x)}{|M(x-y, y) y|^{n+2s}} \,dy\right) \,dx\\
&\quad\le \lim_{s\searrow 0} \frac{c_{n, s}}{2\beta^{n+2s}} \left(\frac{\omega_{n-1}}{n+4s}\right)^{\frac12} \|u\|_{L^2(\R^n)} \|\varphi\|_{L^1(\R^n)} =0.
\end{split}
\]
In the same way, by using the change of variables $z=x-y$, one can prove that
\begin{equation}\label{claimcrtfyvgubhijnòom34567}
\lim_{s\searrow 0} \frac{c_{n, s}}{2} \int_{\R^n}\left(\,\int_{\R^n\setminus B_1} \frac{u(x)\varphi(x-y)}{|M(x-y, y) y|^{n+2s}} \,dy\right) \,dx =0.
\end{equation}

We now claim that
\begin{equation}\label{claimcrtfyvgubhijnòom3456700}
\begin{split}
&\lim_{s\searrow 0} \frac{c_{n, s}}{2} \int_{\R^n}\left(\,\int_{\R^n\setminus B_1} \frac{u(x) \varphi(x)}{|M(x-y, y) y|^{n+2s}}\, dy\right) \,dx\\
&= \lim_{s\searrow 0} \frac{c_{n, s}}{2} \int_{\R^n}\left(\,\int_{\R^n\setminus B_1} \frac{u(x-y) \varphi(x-y)}{|M(x-y, y) y|^{n+2s}} \,dy\right) \,dx\\
&= \frac{1}{2\omega_{n-1}} \int_{\R^n}u(x)\varphi(x) \int_{\mathcal S^{n-1}}\frac{d\mathcal H^{n-1}_\psi}{|M^\infty(x, \psi)\,\psi|^n} \,dx.
\end{split}
\end{equation}
Indeed, by using the change of variables  $r:=\rho^{2s}$, we have
\[
\begin{split}
&\int_{\R^n}\left(\,\int_{\R^n\setminus B_1} \frac{u(x) \varphi(x)}{|M(x-y, y) y|^{n+2s}} \,dy\right) \,dx = \int_{\R^n}u(x)\varphi(x)\left(\,\int_{\R^n\setminus B_1} \frac{dy}{|M(x-y, y) y|^{n+2s}}\right) \,dx\\
&\quad= \int_{\R^n}u(x)\varphi(x)\left(\,\int_1^{+\infty}\int_{\mathcal S^{n-1}}\frac{\rho^{-1-2s} \,d\rho}{|M(x-\rho\psi, \rho\psi)\,\psi|^{n+2s}} \,d\mathcal H^{n-1}_\psi\right) \,dx\\
&\quad=\frac{1}{2s}\int_{\R^n}u(x)\varphi(x)\left(\,\int_1^{+\infty}\int_{\mathcal S^{n-1}}\frac{dr\,d\mathcal H^{n-1}_\psi}{r^2|M(x-r^{\frac{1}{2s}}\psi, r^{\frac{1}{2s}}\psi)\,\psi|^{n+2s}}\right) \,dx.
\end{split}
\]
We notice that, by~\eqref{matdefpos} and recalling that $u\in H^{\overline s}(\R^n)$, we have that
\[
\frac{u(x)\varphi(x)}{r^2|M(x-r^{\frac{1}{2s}}\psi, r^{\frac{1}{2s}}\psi)\,\psi|^{n+2s}}\le \frac{|u(x)\varphi(x)|}{r^2 \beta^{n+2s}}\in L^1\Big(\R^n\times(1,+\infty)\times\mathcal S^{n-1}; dx\,dr\,d\mathcal H^{n-1}_\psi\Big)
\]
Hence, in light of Lemma~\ref{lemmamatrice} (here with $R=1$), we have
\begin{equation}\label{tcfyvigubhoijnpkmolp,ò.}
\begin{split}
&\lim_{s\searrow 0} \frac{c_{n, s}}{2}\int_{\R^n}u(x)\varphi(x)\left(\;\int_{\R^n\setminus B_1} \frac{dy}{|M(x-y, y)\, y|^{n+2s}}\right)\,dx\\
&\quad= \frac{1}{2\omega_{n-1}} \int_{\R^n}u(x)\varphi(x) \int_{\mathcal S^{n-1}}\frac{d\mathcal H^{n-1}_\psi}{|M^\infty(x, \psi)\,\psi|^n} \,dx.
\end{split}
\end{equation}
Moreover,  the change of variables $z:=x-y$ leads to
\[
\int_{\R^n}\int_{\R^n\setminus B_1} \frac{u(x-y) \varphi(x-y)}{|M(x-y, y) y|^{n+2s}} \,dx\, dy =\int_{\R^n} u(z) \varphi(z)\left(\,\int_{\R^n\setminus B_1} \frac{dy}{|M(z, y) y|^{n+2s}}\right) \,dz.
\]
{F}rom this and assumption~\eqref{structuralM}, setting $t:=-y$ we deduce that
\[
\begin{split}
&\int_{\R^n}\int_{\R^n\setminus B_1} \frac{u(x-y) \varphi(x-y)}{|M(x-y, y) y|^{n+2s}} \,dx\, dy =\int_{\R^n} u(z) \varphi(z)\left(\,\int_{\R^n\setminus B_1} \frac{dy}{|M(z+y, -y) y|^{n+2s}}\right) \,dz\\
&\qquad=\int_{\R^n} u(z) \varphi(z)\left(\,\int_{\R^n\setminus B_1} \frac{dy}{|M(z+y, -y)(-y)|^{n+2s}}\right) \,dz\\
&\qquad=\int_{\R^n} u(z) \varphi(z)\left(\,\int_{\R^n\setminus B_1} \frac{dt}{|M(z-t, t) t|^{n+2s}}\right) \,dz.
\end{split}
\]

Bearing in mind~\eqref{tcfyvigubhoijnpkmolp,ò.}, we thus obtain the claim in~\eqref{claimcrtfyvgubhijnòom3456700}, as desired.

Accordingly, combining~\eqref{ygiuhoijnpmkoèlp,ò22}, \eqref{limites000I1}, \eqref{claimcrtfyvgubhijnòom}, \eqref{claimcrtfyvgubhijnòom34567} and~\eqref{claimcrtfyvgubhijnòom3456700}, we obtain the desired result in~\eqref{limitesto0distr}.
\end{proof}

\end{appendix}

\section*{Acknowledgements} 
DA is supported by Ministerio of Ciencia e Innovación (Spain) and European Regional Development Fund (ERDF) PID2021-122122NB-I00 and Junta de Andalucìa FQM-116.

SD,  CS and EV are members of the Australian Mathematical Society (AustMS). CS and EPL are members of the INdAM--GNAMPA.

CS acknowledges the support of the INdAM-GNAMPA through the grant ``Borse di studio per l'estero", of which she is the recipient.

CS thanks the Departamento de An\'alisis Matem\'atico of the Universidad de Granada (where part of this work was carried out) for the warm and generous hospitality.

This work has been supported by the Australian Laureate Fellowship FL190100081 and by the Australian Future Fellowship
FT230100333.

\vfill

\end{document}